\documentclass{article}

\usepackage{amssymb,amsmath,mathrsfs,amsthm,mathtools,amsfonts,esint,comment}
\usepackage{natbib}
\usepackage{caption}
\usepackage{xcolor}
\usepackage{geometry}
 \newtheorem{theorem}{Theorem}[section]
 \newtheorem{corollary}[theorem]{Corollary}
 \newtheorem{lemma}[theorem]{Lemma}
 
 \newtheorem{proposition}[theorem]{Proposition}
 \theoremstyle{definition}
 
 \theoremstyle{remark}
 \newtheorem{remark}[theorem]{Remark}
 
 \numberwithin{equation}{section}

\def \no#1#2#3 {{\bf #1} (#3), #2.}
\def \eds#1#2#3 {#1, #2, #3.}

\def\e{{\rm e}}
\def\d{{\rm d}}

\def\:{{\colon}}

\def\be#1{\begin{equation}\label{#1}}
\def\ee{\end{equation}}

\def\<{\langle}
\def\>{\rangle}
\def\coloneqq{:=}

\newcommand{\p}{\partial}

\newcommand{\supp}{\mathrm{supp}}

\newcommand{\ZZ}{\mathbb{Z}}
\newcommand{\NN}{\mathbb{N}}
\newcommand{\RR}{\mathbb{R}}
\newcommand{\B}{\mathcal{B}}
\newcommand{\C}{\mathcal{C}}

\newcommand{\eqnb}{\begin{equation}}
\newcommand{\eqne}{\end{equation}}

\newcommand\blfootnote[1]{%
  \begingroup
  \renewcommand\thefootnote{}\footnote{#1}%
  \addtocounter{footnote}{-1}%
  \endgroup
}

\begin{document}
\title{Partial regularity of Leray-Hopf weak solutions to the incompressible Navier--Stokes equations with hyperdissipation}
\author{Wojciech S. O\.za\'nski}
\maketitle
\blfootnote{\noindent Institute for Advanced Study, Princeton, NJ 08540, USA \\ wojciech@ias.edu}



\begin{abstract}
We show that if $u$ is a Leray-Hopf weak solution to the incompressible Navier--Stokes equations with hyperdissipation $\alpha \in (1,5/4)$ then there exists a set $S\subset \RR^3$ such that $u$ remains bounded outside of $S$ at each blow-up time, the Hausdorff dimension of $S$ is bounded above by $ 5-4\alpha $ and its box-counting dimension is bounded by $(-16\alpha^2 + 16\alpha +5)/3$. Our approach is inspired by the ideas of Katz \& Pavlovi\'c (\emph{Geom. Funct. Anal., 2002}).
\end{abstract}

\section{Introduction}

We are concerned with the incompressible Navier--Stokes equations with hyper-dissipation,
\eqnb\label{NSE_intro}\begin{split}
u_t + (-\Delta )^\alpha u + (u\cdot \nabla )u + \nabla p &=0\qquad \text{ in } \RR^3,\\
\mathrm{div}\, u &=0
\end{split}
\eqne
where $\alpha \in (1,5/4)$. The equations are equipped with an initial condition $u(0)=u_0$, where $u_0$ is given. We note that the symbol $(-\Delta )^\alpha$ is defined as the pseudodifferential operator with the symbol $(2\pi )^{2\alpha } |\xi |^{2\alpha } $ in the Fourier space, which makes \eqref{NSE_intro} a system of pseudodifferential equations.

It is well-known that the hyperdissipative Navier-Stokes equations \eqref{NSE_intro} are globally well-posed for $\alpha \geq 5/4$, which was proved by \cite{lions}. The question of well-posedness for $\alpha <5/4$, including the case $\alpha =1$ of the classical Navier-Stokes equations, remains open.

The first partial regularity result for the hyperdissipative \eqref{NSE_intro} model was given by \cite{katz_pavlovic}, who proved that the Hausdorff dimension of the singular set in space at the first blow-up time of a local-in-time strong solution is bounded by $5-4\alpha$, for $\alpha \in (1,5/4)$. Recently \cite{colombo_dl_m} showed that if $\alpha\in (1,5/4]$, $u$ is a suitable weak solution of \eqref{NSE_intro} on $\RR^3 \times (0,\infty )$ and 
\[S'\coloneqq \{   (x,t) \colon u \text { is unbounded in every neighbourhood of } (x,t) \} \] denotes the singular set in space-time
then $\mathcal{P}^{5-4\alpha } (S') =0$, where $\mathcal{P}^s$ denotes the $s$-dimensional parabolic Hausdorff measure. This is a stronger result than that of \cite{katz_pavlovic} since it is concerned with space-time singular set $S'$ (rather than the singular set in space at the first blow-up), it is a statement about the Hausdorff measure of the singular set (rather than merely the Hausdorff dimension) and it includes the case $\alpha =5/4$ (in which case the statement, $\mathcal{P}^0 (S') =0$, means that the singular set is in fact empty, and so \eqref{NSE_intro} is globally well-posed). The main ingredient of the notion of a ``suitable weak solution'' in the approach of \cite{colombo_dl_m} is a local energy inequality, which is a generalisation of the classical local energy inequality in the Navier--Stokes equations (i.e. when $\alpha =1$) to the case $\alpha \in (1,5/4)$. The fractional Laplacian $(-\Delta )^{\alpha }$ is incorporated in the local energy inequality using a version of the extension operator introduced by \cite{caffarelli_silvestre} (see also \cite{yang_ext} and Theorem 2.3 in \cite{colombo_dl_m}). \cite{colombo_dl_m} also show a bound on the box-counting dimension of the singular set 
\eqnb\label{colombo_db_bound}
d_B(S' \cap (\RR^3 \times [t, \infty ) )) \leq (-8\alpha^2 - 2 \alpha + 15 )/3 
\eqne
for every $t>0$. Note that this bound reduces to $0$ at $\alpha =5/4$ and converges to $5/3$ as $\alpha \to 1^+$, which is the bound that one can deduce from the classical result of \citep{CKN}, see \cite{rob_sad_2007} or Lemma 2.3 in \cite{ozanski_book} for a proof. We note that this bound (for the Navier--Stokes equations) has recently been improved by \cite{wang_yang} (to the bound $d_B(S)\leq 7/6$).

Here, we build on the work of \cite{katz_pavlovic}, as their ideas offer an entirely different viewpoint on the theory of partial regularity of the Navier--Stokes equations (or the Navier--Stokes equations with hyper- and hypo- dissipation), as compared to the early work of Scheffer (1976\emph{a}, 1976\emph{b}, 1977, 1978 \& 1980)\nocite{scheffer_hausdorff_measure}\nocite{scheffer_partial_reg}\nocite{scheffer_turbulence}\nocite{scheffer_dim_4}\nocite{scheffer_NSE_on_bdd_domain}, the celebrated result of \cite{CKN}, as well as alternative approaches of \cite{vasseur_2007}, \cite{lin}, \cite{ladyzhenskaya_seregin}, as well as numerous extensions of the theory, such as \cite{colombo_dl_m} and \cite{tang_yu}. Instead it is concerned with the dynamics (in time) of energy packets that are localised both in the frequency space and the real space $\RR^3$, and with studying how do these packets move in space as well as transfer the energy between the high and low frequencies. An important concept in this approach is the so-called \emph{barrier} (see \eqref{barrier_property}), which, in a sense, quarantines a fixed region in space in a way that prevents too much energy flux entering the region. This property is essential in showing regularity at points outside of the singular set.

In order to state our results, we will say that $u$ is a (global-in-time) \emph{Leray-Hopf weak solution} of \eqref{NSE_intro} if 
\begin{enumerate}
\item[(i)] it satisfies the equations in a weak sense, namely
\[
\int_0^t \int \left( -u \, \varphi_t + (-\Delta )^{\alpha/2 } u \cdot (-\Delta )^{\alpha/2 } \varphi + (u\cdot \nabla )u \cdot \varphi \right) = \int u_0 \cdot \varphi - \int u(t) \cdot \varphi (t)
\]
holds for all $t>0$ and all $\varphi \in C_0^\infty ([0,\infty )\times \RR^3 ; \RR^3 )$ with $\mathrm{div}\,\varphi (s)=0$ for all $s\geq 0$ (where we wrote $\int \equiv \int_{\RR^3}$ for brevity),
\item[(ii)] the strong energy inequality,
\eqnb\label{EI_prelims}
\frac{1}{2} \| u(t) \|^2 + \int_s^t \| (-\Delta)^{\alpha/2} u (\tau ) \|^2 \d \tau \leq \frac{1}{2} \| u (s) \|^2
\eqne
holds for almost every $s\geq 0$ (including $s=0$) and every $t>s$.
\end{enumerate}
We note that Leray--Hopf weak solutions admit \emph{intervals of regularity}; namely for every Leray-Hopf weak solution there exists a family of pairwise disjoint intervals $(a_i ,b_i)\subset (0,\infty )$ such that $u$ coincides with some strong solution of \eqref{NSE_intro} on each interval and
\eqnb\label{int_of_reg}
\mathcal{H}^{(5-4\alpha )/2\alpha }\left( \RR \setminus \bigcup_i (a_i,b_i ) \right) =0,
\eqne
see Theorem 2.6 and Lemma 4.1 in \cite{jiu_wang} for a proof. This is a generalisation of the corresponding statement in the case $\alpha =1$ (i.e. in the case of the Navier--Stokes equations), see Section 6.4.3 in \cite{ozanski_pooley} and Chapter 8 in \cite{NSE_book}.\\
 
Given $u_0 \in L^2 (\RR^3 )$ with $\mathrm{div}\, u_0 =0$ there exists at least one global-in-time Leray-Hopf weak solution (see Theorem 2.2 in \cite{colombo_dl_m}, for example). We denote by $S$ the singular set in space of $u$ at single blow-up times, namely 
\eqnb\label{def_singset}
S\coloneqq \bigcup_i S_i,
\eqne
where
\[
S_i \coloneqq \{ x\in \RR^3 \colon u \text{ is unbounded in } U\times ((a_i+b_i)/2, b_i  ) \text{ for any neighbourhood } U\text{ of }x\}
\]
denotes the singular-set 
In particular, if $x\not \in S$ then $\limsup_{t\to b_i^-} \| u(t) \|_{L^\infty (U)} \leq c_i$ for every $i$ and $U\ni x$. The first of our main results is the following. 
\begin{theorem}\label{thm_main}
Let $u$ be a Leray-Hopf weak solution of \eqref{NSE_intro} with $\alpha \in (1,5/4)$ and an initial condition $u_0\in H^1(\RR^3)$, and let $\varepsilon >0$. There exists $C>0$ and a family of collections $\B_j$ of cubes $Q\subset \RR^3$ of sidelength $2^{-j(1+\varepsilon )}$ such that 
\[ \# \B_j \leq C\,2^{j(5-4\alpha +\varepsilon )}
\]
for each $j\in \NN$, and
\eqnb\label{claim_of_thm1}
S \subset \limsup_{j\to \infty } \bigcup_{Q\in \B_j } Q.
\eqne
In particular, $d_H(S)\leq 5-4\alpha $.
\end{theorem}
Here $d_H$ stands for the Hausdorff dimension, and we recall that $\limsup_{j\to \infty } G_j \coloneqq \cap_{k\geq 0} \cup_{j\geq k} G_j$ denotes the set of points belonging to infinitely many $G_j$'s. It is well-known (see Lemma~3.1 in \cite{katz_pavlovic}, for example) that \eqref{claim_of_thm1} implies that $d_H (S) \leq 5-4\alpha + \varepsilon$, from which the last claim of the theorem follows by sending $\varepsilon>0$.

We note that $C$ might depend on $\varepsilon$, but it does not depend on the interval of regularity $(a_i,b_i)$, which gives us a control of the structure of the singular sets $S_i$ that is uniform across blow-ups in time of a Leray-Hopf weak solution. This is an improvement of the result of \cite{katz_pavlovic}, who obtained such control for a given strong solution, and so for each interval of regularity $(a_i, b_i)$ of a Leray-Hopf weak solution their result implies existence of $C_i>0$ such that $S_i \subset \limsup_{j\to \infty } \bigcup_{Q\in \B^{(i)}_j } Q   $ for some collections $\B^{(i)}_j$ of cubes of sidelength $2^{-j(1+\varepsilon )}$ satisfying $\B^{(i)}_j \leq C_i\,2^{j(5-4\alpha +\varepsilon )}$ for all $j$. One could therefore expect that the constants $C_i$ become unbounded as $i$ varies (for example in a scenario of a limit point of the set of blow-up times $\{b_i\}$), and Theorem~\ref{thm_main} shows that it does not happen. 

We note however, that Theorem \ref{thm_main} does not estimate the dimension of the singular set at the blow-up time which is not an endpoint $b_i$ of an interval of regularity (but instead a limit of a sequence of such $b_i$'s). In other words, if $x\not \in S$, $U\ni x$ is a small open neighbourhood of $x$ and $\{ (a_i,b_i)\}_{i}$ is a collection of consecutive intervals of regularity of $u$, we show that $\sup_{U\times ((a_i+b_i)/2,b_i)} |u| = c_i<\infty $, but our result does not exclude the possibility that $c_i\to \infty $ as $i\to \infty$. It also does not imply boundedness of $|u(t)|$ at times $t$ near the left endpoint $a_i$ of any interval of regularity $(a_i,b_i)$. These issues are related to the fact that inside the barrier we still have to deal with infinitely many energy packets (i.e. infinitely many frequencies and cubes in $\RR^3$). Thus, supposing that the estimate on the the energy packets inside the barrier breaks down at some $t$ we are unable to localize the packet (i.e. the frequency and the cube) on which the growth occurs near $t$, unless $t$ is located inside an interval of regularity, see Step~1 of the proof of Theorem \ref{thm_reg_outside_E} for details.

The proof of Theorem \ref{thm_main} is inspired by the strategy of the proof of \cite{katz_pavlovic}, which we extend to the case of Leray-Hopf weak solutions and we use a more robust main estimate. The main estimate is an estimate on the time derivative of the $L^2$ norm of Littlewood--Paley projection $P_ju$ combined with a cut-off in space (the \emph{energy packet}), see \eqref{basic_est}. We show that such norm is continuous in time (regardless of putative singularities of a Leray-Hopf weak solution), which makes the main estimate valid for all $t>0$. Inspired by \cite{katz_pavlovic}, we then define \emph{bad cubes} and \emph{good cubes} (see \eqref{j-good}) and show that we have a certain more-than-critical decay on a cube that is good and has some good ancestors. We then construct $\B_j$ as a certain cover of bad cubes and prove \eqref{claim_of_thm1}. 

Our second main result is concerned with the box-counting dimension. We let 
\eqnb\label{def_singset_k}
S^{(k)} \coloneqq \bigcup_{i\leq k } S_i 
\eqne
\begin{theorem}\label{thm_main2}
Let $u$ be as in Theorem \ref{thm_main}. Then $d_B(S^{(k)})\leq (-16\alpha^2 +16\alpha +5 )/3$ for every $k\in \NN$.
\end{theorem}
We prove the theorem by sharpening the argument outlined below Theorem~\ref{thm_main}. We recall that the box-counting dimension $d_B$ is concerned with covering the given set by a collection of balls of radius $r$,
\eqnb\label{def_db}
d_B(K) \coloneqq \limsup_{r\to 0 } \frac{\log N(K,r)}{-\log r},
\eqne
where $N (K,r)$ denotes the minimal number of balls (or boxes) of radius $r$ required to cover $K$. In this context, one can actually use the families $\B_j$ from  \eqref{claim_of_thm1} to deduce that  $d_B(S^{(k)} )\leq (-64\alpha^3 +96\alpha^2 - 48 \alpha +35 )/9$ for every $k$, which we discuss in detail in Section~\ref{sec_db}. This is however a worse estimate than claimed in Theorem~\ref{thm_main2}. 

In fact, in Section~\ref{sec_db} we improve this estimate by constructing refined families $\C_j$ that, in a sense, give a more robust control of the low modes, which reduces the number of cubes required to cover the singular set and hence improve the bound on $d_B(S^{(k)})$. See the informal discussion following Proposition~\ref{prop_db} for more insight about this improvement.

We note that we can only estimate $d_B(S^{(k)})$ (rather than $d_B(S)$) because of the localisation issue described above. To be more precise, for each sufficiently small $\delta >0$ we can construct a family of cubes of sidelength $\delta >0$ that covers the singular set when $t$ approaches a singular time, and that has cardinality less than or equal $\delta^{(-16\alpha^2 +16\alpha +5 )/3 +\varepsilon}$ for any given~$\varepsilon >0$. This family can be constructed independently of the interval of regularity, but given $x$ outside of this family we can show that the solution is bounded in a neighbourhood of $x$ if the choice of (sufficiently small) $\delta $ is dependent on the interval of regularity. This gives the limitation to only finite number of intervals of regularity in the definition of $S^{(k)}$.


We note that the result of \cite{colombo_dl_m} is stronger than our result in the sense that it is concerned with the space-time singular set $S'$ (rather than singular set $S$ in space), it is concerned with the parabolic Hausdorff measure of $S'$ (rather than merely the bound on $d_H(S')$) and its estimate of $d_B(S')$ is sharper than our estimate on $d_B(S_k)$.

However, our result is stronger than \cite{colombo_dl_m} in the sense that it applies to any Leray--Hopf weak solutions (rather than merely suitable weak solutions). In other words we do not use the local energy inequality, which is the main ingredient of \cite{colombo_dl_m}. Also, our approach does not include any estimates of the pressure function. In fact we only consider the Leray projection of the first equation in \eqref{NSE_intro}, which eliminates the pressure. Furthermore, our approach can be thought of as an extension of the global regularity of \eqref{NSE_intro} for $\alpha >5/4$. In fact, the following corollary can be proved almost immediately using our main estimate, see Section \ref{sec_reg_alpha_large}.
\begin{corollary}\label{cor_alpha_big_and_higher_dim}
If $\alpha >5/4$ then \eqref{NSE_intro} is globally well-posed.
\end{corollary}

We also point out that our estimate on the box-counting dimension, $d_B(S_k)\leq (-16\alpha^2 +16\alpha +5 )/3$, converges to $5/3$ as $\alpha \to 1^+$, same as \eqref{colombo_db_bound}.

Finally, we also correct a number of imprecisions appearing in \cite{katz_pavlovic}, see for example Remark \ref{rem2} and Step~1 of the proof of Theorem \ref{thm_reg_outside_E}.

The structure of the article is as follows. In Section \ref{sec_prelims} we introduce some preliminary concepts including the Littlewood-Paley projections, paraproduct decomposition, Bernstein inequalities as well as a number of tools that allow us to manipulate quantities involving cut-offs in both the real space and the Fourier space, which includes estimates of the errors when one moves a Littlewood--Paley projection across spatial cut-offs and vice versa. We prove the first of our main results, Theorem \ref{thm_main}, in Section \ref{sec_main_result}. We prove Corollary \ref{cor_alpha_big_and_higher_dim} in Section \ref{sec_reg_alpha_large} and we prove the second of our main results, Theorem \ref{thm_main2}, in Section \ref{sec_db}.

\section{Preliminaries}\label{sec_prelims}
Unless specified otherwise, all function spaces are considered on the whole space $\RR^3$. In particular $L^2 \coloneqq L^2(\RR^3 )$. We do not use the summation convention. We will write $\p_i\coloneqq \p_{x_i}$, $B(R)\coloneqq \{ x\in \RR^3 \colon |x|\leq R \}$,  $\int\coloneqq \int_{\RR^3}$, and $\| \cdot \|_p \coloneqq \| \cdot \|_{L^p (\RR^3 )}$. We reserve the notation ``$\| \cdot \|$'' for the $L^2$ norm, that is $\| \cdot \| \coloneqq \| \cdot \|_2$.

We denote any positive constant by $c$ (whose value may change at each appearance). We point out that $c$ might depend on $u_0$ and $\alpha $, which we consider fixed throughout the article. As for the constants dependent on some parameters, we sometimes emphasize the parameters by using subscripts. For example, $c_{k,q}$ is any constant dependent on $k$ and $q$.

We denote by ``$e(j)$'' (a \emph{$j$-negligible error}) any quantity that can be bounded (in absolute value) by $c_K 2^{-Kj}$ for any given $K>0$. 

We say that a differential inequality $f'\leq g$ on a time interval $I$ is satisfied in the \emph{integral sense} if 
\eqnb\label{integral_sense}
f(t) \leq f(s) + \int_s^t g(\tau ) \d \tau \quad \text{ for every } t,s\in I \text{ with } t>s.
\eqne

We recall that Leray--Hopf weak solutions are weakly continuous with values in $L^2$. Indeed, it follows from part (i) of the definition that
\[
\int u(t) \varphi \quad \text{ is continuous for every }\varphi \in C_0^\infty (\RR^3 )\text{ with } \mathrm{div}\, \varphi =0.
\]
This is also true is $\mathrm{div}\, \varphi \ne 0$, as in this case one can apply Helmholtz decomposition to write $\varphi = \phi + \nabla \psi$, where $\mathrm{div}\,\phi =0 $ (then $\int u(t) \phi $ is continuous and $\int u(t) \nabla \psi =0$ since $u(t)$ is divergence-free). Thus, since part (ii) gives that $\{ u(t) \}_{t\geq 0}$ is bounded in $L^2$, weak continuity of $u(t)$ follows.
\subsection{Littlewood-Paley projections}
Given $f\in L^1 (\RR^3)$ we denote by $\widehat{f}$ its Fourier transform, i.e.
\[
\widehat{f} (\xi ) \coloneqq \int f(x) \e^{-2\pi i x\cdot \xi} \d x,\quad \xi \in \RR^3,
\]
and by $\check{f}$ its inverse Fourier transform, i.e. $\check{f} (x) \coloneqq \widehat{f}(-x)$. Let $h\in C^\infty (\RR ; [0,1])$ be any function such that $h(x)=1 $ for $x<1$ and $h(x)=0$ for $x>2$. We set $p(x) \coloneqq h(|x|)-h(2|x|)$, where $x\in \RR^3$, and we let 
\eqnb\label{def_of_pj}
p_j(\xi ) \coloneqq p(2^{-j}\xi)\quad \text{ for }j\in \ZZ,
\eqne
and we let $P_j$ (the $j$-th Littlewood-Paley projection) be the corresponding multiplier operator, that is
\[
\widehat{P_j f}(\xi ) \coloneqq p_j (\xi ) \widehat{f}(\xi ).
\]
By construction, $\supp\, p_j \subset B(2^{j+1})\setminus B(2^{j-1})$. We note that $\sum_{j\in \ZZ}p_j =1$, and so formally $\sum_{j\in \ZZ} P_j =\mathrm{id}$. We also denote
\eqnb\label{pj_convention}
 \widetilde{P}_j \equiv P_{j\pm 2} \coloneqq \sum_{k=j-2}^{j+2} P_k,\quad P_{j- 4,j+2} \coloneqq \sum_{k=j-4}^{j+2} P_k,\quad P_{\leq j} \coloneqq \sum_{k=-\infty }^j P_k ,\quad P_{\geq j} \coloneqq \sum_{k=j }^\infty P_k ,
\eqne
and analogously for $\widetilde{p}_j$, $p_{j-4,j+2}$, $p_{\leq j}$, $p_{\geq j}$. By a direct calculation one obtains that
\eqnb\label{pj_vs_p_inverse_FT}
\check{p_j} (y) = 2^{3j} \check{p}(2^j y)
\eqne
for all $j\in \ZZ$, $y\in \RR^3$. In particular $\| \check{p_j} \|_1 =c$ and so, since $P_j f = \check{p_j} \ast f$ (where ``$\ast $'' denotes the convolution), Young's inequality for convolutions gives
\eqnb\label{lp_pj_is_bdd_on_Lp}
\| P_j u \|_q \leq c \| u \|_q
\eqne
for any $q\in [1,\infty ]$. Moreover, given $K>0$ there exists $c_K >0$ such that
\eqnb\label{pj_bound}
\left| \check{p_j} (y) \right| \leq c_K (2^j |y|)^{-2K} 2^{3j}
\eqne
and
\eqnb\label{deriv_pj_bound}
\left| \p_i  \check{p_j} (y) \right|\leq c_K (2^j |y|)^{-2K} 2^{4j} 
\eqne
for all $j\in \ZZ$, $y\ne 0$ and $i=1,2,3$. Indeed, the case $j=0$ follows by noting that 
\[
\e^{2\pi i y\cdot \xi } = (-4\pi^2 |y|^2)^{-K} \Delta_{\xi}^K \e^{2\pi i y\cdot \xi }
\]
and calculating
\[\begin{split}
\left| \check{p} (y) \right| &= \left| \int p (\xi ) \e^{2\pi i y\cdot \xi  } \d \xi \right| = (4\pi^2 |y|^2)^{-K}  \left| \int \Delta^K p (\xi) \e^{2\pi i y\cdot \xi  } \d \xi \right| \\
&\leq c_K |y|^{-2K} \int_{B(2)} |\Delta^K p | = c_K |y|^{-2K},
\end{split}
\]
(and similarly $|\p_i \check{p} (y) |\leq c_K |y|^{-2K}$) where we have integrated by parts $2K$ times, and the case $j\ne 0$ follows from \eqref{pj_vs_p_inverse_FT}. Using \eqref{pj_bound} and \eqref{deriv_pj_bound} we also get
\eqnb\label{Lq_norm_of_pj_outside_delta_ball}
\| \check{p_j} \|_{L^q ( B(d )^c)} \leq C_{K,q} (d 2^j)^{-2K +3/q} \,2^{3j(q-1)/q}
\eqne
and
\eqnb\label{Lq_norm_of_deriv_pj_outside_delta_ball}
\| \p_i \check{p_j} \|_{L^q ( B( d )^c)} \leq C_{K,q} (d 2^j)^{-2K +3/q} \,2^{j(1+n(q-1)/q)},
\eqne
respectively, for any $K>0$, $d>0$, $i=1,2,3$, $j\in \ZZ$ and $q\geq 1$. Indeed
\[
\int_{\RR^3\setminus B(d)} |\check{p_j} (y)|^q \d y\leq C_{K,q} 2^{-jq(2K-3)}\int_{|y|\geq d } |y|^{-2Kq}  \d y = C_{K,q} 2^{-jq(2K-3)} d^{-2Kq+3}
\]
from which \eqref{Lq_norm_of_pj_outside_delta_ball} follows (and \eqref{Lq_norm_of_deriv_pj_outside_delta_ball} follows analogously).
We note that the same is true when $p$ is replaced by any compactly supported multiplier.
\begin{corollary}
Let $\lambda \in C_0^\infty (\RR^3)$ and, given $j\in \ZZ$, set $\lambda_j (\xi ) \coloneqq \lambda ( 2^{-j} \xi )$. Then given $d >0$\[
\| \check{\lambda_j} \|_{L^2 (\RR^3\setminus B(d ))} \leq c_K 2^{-j(2K-3)} d^{-2K+3/2}.\]
\end{corollary}


We will denote by $T$ the Leray projection, that is 
\eqnb\label{def_of_T}
\widehat{Tf} (\xi ) \coloneqq \left(I - \frac{\xi \otimes \xi }{|\xi |^2 }  \right) \widehat{f}
\eqne
where $f\colon \RR^3 \to \RR^3$, and $I$ denotes the $3\times 3$ identity matrix.

\subsection{Bernstein inequalities}
Here we point out classical Bernstein inequalities on $\RR^3$:
\eqnb\label{bernstein_ineq_single}
\| P_j f\|_q \leq c\, 2^{3j(1/p-1/q)} \| P_j f\|_p
\eqne
and
\eqnb\label{bernstein_ineq_leq}
\left\| P_{\leq j}  f \right\|_q \leq c\, 2^{3j(1/p-1/q)} \left\|  P_{\leq j} f \right\|_p 
\eqne
for any $1\leq p\leq q \leq \infty $. We refer the reader to Lemma 2.1 of \cite{bahouri_chemin_danchin} for a proof.

\subsection{The paraproduct formula}\label{sec_paraproduct_formula}
Here we concern ourselves with a structure of a Littlewood-Paley projection of a product of two functions, $P_j (fg)$. One could obviously write $f=\sum_{k\in \ZZ} P_kf$ (and similarly for $g$) to obtain that
\eqnb\label{bony_naive}
P_j(fg)= P_j \left( \sum_{k,m\in\ZZ } P_kf\,P_mg\right).
\eqne
However, since functions $p_j, p_k$ have pairwise disjoint supports for many pairs $j,k\in \ZZ$, one could speculate that some of the terms on the right-hand side of \eqref{bony_naive} vanish. This is indeed the case and 
\eqnb\label{bony}
\begin{split}
P_j (fg) &= P_j \left( {P}_{j\pm 2} f\, P_{\leq j-5 } g + P_{\leq j-5} f\,{P}_{j\pm 2} g + P_{j- 4,j+2} f\, P_{j\pm 4 } g +  \sum_{ k\geq j+3}P_{k } f \,P_{k\pm 2} g \right)\\
&= P_j \left( K_{loc,low} + K_{low,loc} + K_{loc} +K_{hh} \right),
\end{split}
\eqne
which is also known as \emph{Bony's decomposition formula}. For the sake of completeness we prove the formula below. Heuristically speaking, $K_{loc,low}$ corresponds to interactions between local (i.e. around $j$) modes of $f$ and low modes of $g$, $K_{low,loc}$ to interactions between low modes of $f$ and local modes of $g$, $K_{loc}$ to local interactions and $K_{hh}$ to interactions between high modes, see Figure \ref{paraproduct_cancellation} for a geometric interpretation of \eqref{bony}. 

We now prove \eqref{bony}. For this it is sufficient to show that 
\eqnb\label{cancellation}
P_j (P_kf\,P_m g) = 0\quad \text{ for } (k,m)\in R_1 \cup R_2 \cup R_3, 
\eqne
where $R_1$, $R_2$, $R_3$ are as sketched on Fig. \ref{paraproduct_cancellation}. The Fourier transform of $w\coloneqq P_j (P_kf\,P_m g)$ is
\[
\widehat{w}(\xi ) = p_j (\xi ) \int p_k (\eta ) \widehat{f} (\eta ) p_m (\xi-\eta ) \widehat{g} (\xi-\eta ) \d \eta
\]
We can assume that $|\xi |\in (2^{j-1},2^{j+1})$ (as otherwise $p_j (\xi)$ vanishes) and that $|\eta |\in (2^{k-1},2^{k+1})$ (as otherwise $p_k (\eta )$ vanishes).
\begin{enumerate}
\item[\emph{Case 1.}] $(k,m)\in R_1$. Suppose that $k\geq m$ (the opposite case is analogous). Then $j\geq k+3$ (see Fig. \ref{paraproduct_cancellation}) and so
\[
|\xi - \eta | \geq |\xi | - |\eta | \geq 2^{j-1} - 2^{k+1} \geq 2^{k+2}- 2^{k+1} = 2^{k+1} \geq 2^{m+1}.
\]
Thus $p_m (\xi-\eta )$ vanishes.
\item[\emph{Case 2.}] $(k,m)\in R_2 \cup R_3$. Suppose that $(k,m)\in R_2$ (the case $(k,m)\in R_3$ is analogous). Then $m\geq k+3$ and $m\geq j+3$ (see Fig. \ref{paraproduct_cancellation}) and so
\[
|\xi - \eta | \leq  |\xi | + |\eta | \leq 2^{j+1} + 2^{k+1} \leq  2\cdot  2^{m-2}  =2^{m-1}.
\]
Hence $p_m (\xi-\eta )$ vanishes as well, and so \eqref{cancellation} follows.
\end{enumerate}
\begin{figure}[h]
\centering
\captionsetup{width=.9\linewidth}
 \includegraphics[width=0.8\textwidth]{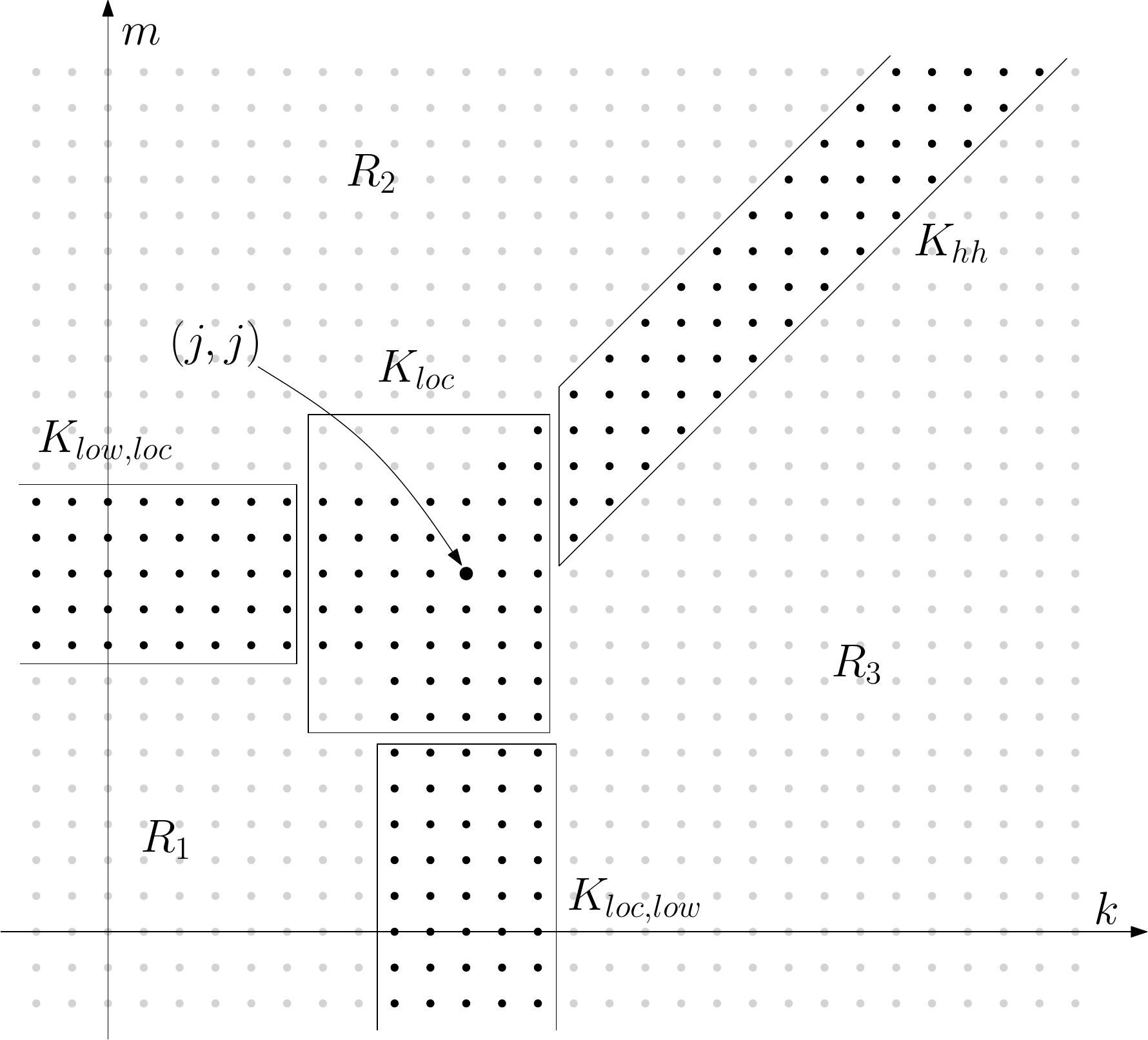}
 \nopagebreak
  \captionof{figure}{Sketch of the interpretation of the terms on the right-hand side of \eqref{bony}. The regions $R_1$, $R_2$, $R_3$ (consisting of grey dots) correspond to pairs $(k,m)$ for which $P_j(P_kf\,P_mg)$ vanishes, see the discussion following \eqref{cancellation}.}\label{paraproduct_cancellation} 
\end{figure}

\subsection{Moving bump functions across Littlewood--Paley projections}\label{sec_moving_bumps_fcns}
Here we show the following
\begin{lemma}\label{lem_moving_bump_functions_across_LP}
Let $\phi_1, \phi_2 \colon \RR^3 \to [0,1]$ be such that their supports are separated by at least $d >2^{-j}$. Then 
\[
\| \phi_1 P_j (\phi_2 f) \|_q \leq c_K (d2^j )^{-2K+3}\| f \|_q
\]
for all $q\in [1,\infty ]$, $j\in \ZZ$, $K>0$ and $f\in L^q (\RR^3)$. Furthermore, if $| \nabla \phi_2 |\leq c\, d^{-1}$ then
\[
\| \phi_1 P_j (\phi_2 \nabla f) \|_q \leq c_K (d2^j )^{-2K+3} 2^{j} \| f \|_q
\]
\end{lemma}
We will only use the lemma (and the corollary below) with $q=2$ or $q=1$.
\begin{proof}
We note that
\eqnb\label{moving_bumps_temp}\begin{split}
\phi_1 P_j (\phi_2 f) (x) &= \phi_1 (x) \int_{\mathrm{supp}\,\phi_2} \check{p_j}(x-y) \phi_2 (y) f(y) \d y\\
&= \phi_1 (x) \int_{\mathrm{supp}\,\phi_2} \chi_{|x-y|>d }\, \check{p_j}(x-y) \phi_2 (y) f(y) \d y
\end{split}
\eqne
since the supports of $\phi_1$, $\phi_2$ are at least $d $ apart. Thus using Young's inequality for convolutions
\[
\| \phi_1 P_j (\phi_2 f) \|_q \leq \| \check{p_j} \|_{L^1 (B(d )^c)} \| \phi_2 f \|_q \leq c_K (d 2^j)^{-2K + 3}  \| f \|_q
\]  
for any $K>0$, where we used \eqref{Lq_norm_of_pj_outside_delta_ball}. This shows the first claim of the lemma. The second claim follows by replacing $f$ by $\nabla f$ in \eqref{moving_bumps_temp}, integrating by parts, and using Young's inequality for convolutions to give
\[\begin{split}
\| \phi_1 P_j (\phi_2 \nabla f) \|_q &\leq c \| \nabla \check{p_j} \|_{L^1 (B(d )^c)} \| \phi_2 f \|_q + \|  \check{p_j} \|_{L^1 (B(d )^c)} \| \nabla \phi_2 f \|_q   \\
&\leq c_K (d 2^j)^{-2K + 3} 2^{j}  \| f \|_q,
\end{split}
\]  
where we also used the assumption that $|\nabla \phi_2 |\leq c\, d^{-1} < c \, 2^{j}$.
\end{proof}
In fact the same result is valid when $P_j$ is replaced by the composition of $P_j$ with any $0$-homogeneous multiplier (e.g. the Leray projector).
\begin{corollary}\label{cor_moving_bump_functions_across_LP+any_homog_mult}
Let $M$ be a bounded, $0$-homogeneous multiplier (i.e. $\widehat{Mf}(\xi )= m (\xi ) \widehat{f}(\xi )$, where $m(\lambda \xi )=m(\xi )$ for any $\lambda >0$). Let $\phi_1, \phi_2 \colon \RR^3 \to [0,1]$ be such that their supports are separated by at least $d >2^{-j}$. Then 
\[
\| \phi_1 M P_j (\phi_2 \nabla  f) \|_q \leq c_K  (d2^j )^{-2K+3 }  2^j \| f \|_q
\]
for all $q\in[1,\infty ]$, $j\in \ZZ$, $K>0$ and $f\in L^q (\RR^3)$.
\end{corollary}

\subsection{Moving Littlewood-Paley projections across spatial cut-offs}
We say that $\phi \in C_0^\infty (\RR^3)$ is a $d $\emph{-cutoff} if $\mathrm{diam}(\mathrm{supp}\, \phi )\leq c\, d$ and $|D^l \phi |\leq c_l d^{-l}$ for any $l\geq 0$.

We denote by $e_d (j)$ any quantity that can be bounded (in absolute value) by $c_K 2^{cj} (d2^{j})^{c-K}$ for any given $K>0$. The point of such a bound is that it will articulate the the dependence of the size of the error in our main estimate (see Proposition \ref{prop_basic_estimate}) on both $j$ and $d$.

In this section we show that, roughly speaking, we can move Littlewood-Paley projections $P_j$ across $d $-cutoffs as long as $d > 2^{-j}$
\begin{lemma}\label{lem_moving_lp_across_bumps}
Given a $d$-cutoff $\phi$, $q\in [1,\infty ]$ and multiindices $\alpha, \beta $, with $|\beta |, |\alpha |\leq 3$, 
\[
\| (1-\widetilde{P}_j) D^\alpha (\phi P_j D^\beta f) \|_q \leq   e_d (j) \| f \|_q 
\]
for every $j$.
\end{lemma}
\begin{proof} We write $\phi = \phi_1 + \phi_2$, where
\[
\begin{split}
\widehat{\phi_1} (\xi ) & \coloneqq \chi_{|\xi |\leq 2^{j-2}} \widehat{ \phi }(\xi ), \\
\widehat{\phi_2} (\xi ) & \coloneqq \chi_{|\xi |> 2^{j-2}} \widehat{ \phi }(\xi ).
\end{split}
\]
Note that 
\[
\widehat{(\phi_1 P_j D^\beta f) } (\xi ) = \int \widehat{\phi_1} (\xi - \eta ) p_j (\eta )  (2\pi i )^{|\beta |} \eta^\beta \widehat{f} (\eta ) \, \d \eta 
\]
is supported in $|\xi |\in (2^{j-2},2^{j+2})$ (as $\widehat{\phi_1} (\xi - \eta )$ is supported in $\{ |\xi - \eta |\leq 2^{j-2} \}$ and $p_j (\eta )$ is supported in $\{ 2^{j-1} < |\eta | < 2^{j+1} \}$). Since $\widetilde{p}_j(\xi)=1$ for such $\xi$ we obtain
\eqnb\label{phi1_is_ok}
\phi_1 P_j D^\beta f = \widetilde{P}_j \phi_1 P_j D^\beta f,
\eqne
and so it suffices to show that
\[
\| (1-\widetilde{P}_j ) D^\alpha (\phi_2 P_j D^\beta f ) \|_q \leq e_d (j) \| f \|_q
\]
We will show that 
\eqnb\label{will_show_this_1}
\| \widehat{ D^\alpha  \phi_2 } \|_1 \leq e_d (j)
\eqne
for every $|\alpha |\leq 3$. Then the claim follows by writing
\[
\begin{split}
\| (1-\widetilde{P}_j )D^\alpha  (\phi_2 P_j D^\beta f ) \|_q &  \leq \sum_{\alpha_1+\alpha_2 = \alpha}\| D^{\alpha_1}\phi_2 P_j D^{\alpha_2+\beta} f \|_q \\
&\leq \sum_{\alpha_1+\alpha_2 = \alpha}  \| D^{\alpha_1}\phi_2 \|_\infty \| P_j D^{\alpha_2+\beta} f \|_q \\
&\leq \sum_{|\alpha_1|\leq 3}\| \widehat{ D^{\alpha_1} \phi_2 } \|_1  \cdot 2^{6j} \| f \|_q \\
&\leq   e_d (j) \|f  \|_q.
\end{split}
\]

In order to see \eqref{will_show_this_1} we first note that
\[\begin{split}
| \widehat{D^\alpha \phi_2 } (\xi ) | &\leq c |\xi |^{|\alpha |} \left| \int \phi_2 (x) \mathrm{e}^{-2\pi \mathrm{i} x\cdot \xi } \, \d x\right|\\
&= c |\xi |^{|\alpha |}(4\pi^2 |\xi |^2 )^{-K} \left| \int \phi_2 (x) \Delta^K \mathrm{e}^{-2\pi \mathrm{i} x\cdot \xi } \, \d x\right| \\
&= c |\xi |^{|\alpha |}(4\pi^2 |\xi |^2 )^{-K} \left| \int \Delta^K \phi_2 (x)  \mathrm{e}^{-2\pi \mathrm{i} x\cdot \xi } \, \d x\right| \\
&\leq c_K |\xi |^{-2K+|\alpha |}  d^{-2K+3}
\end{split} 
\]
Thus
\[\begin{split}
\| \widehat{D^\alpha \phi_2} \|_1 &=c \int_{|\xi |>2^{j-2}} | \widehat{D^\alpha \phi_2 } (\xi )| \\
&\leq c_K d^{-2K+3}  \int_{|\xi |>2^{j-2}} |\xi |^{-2K+|\alpha |} \\
&=  c_K 2^{3j}(d 2^j )^{-2K+3},
\end{split}
\]
which gives \eqref{will_show_this_1}.
\end{proof}
Similarly one can put the Littlewood-Paley projection ``inside the cutoff''. In this case one can prove a similar statement as in Lemma \ref{lem_moving_lp_across_bumps}, but, since we will only need a simple version with no derivatives, we state a simplified statement.
\begin{corollary}\label{cor_put_Pj_inside}
Given a $d$-cutoff $\phi$ $\| P_j (\phi (1-{{P}}_{j\pm 2})f ) \| \leq e_d (j)$ for every $j$.
\end{corollary}
\begin{proof}
The claim follows using the same decomposition as above, $\phi=\phi_1+\phi_2$. Since 
\[
\widehat{\phi (1-P_{j\pm 2})f } (\xi) = \int \widehat{\phi_1} (\xi - \eta ) (1-p_{j\pm 2} (\eta )) \widehat{f} (\eta ) \d \eta
\]
we see that (since $|\eta |\in (-\infty , 2^{j-2})\cup (2^{j+2},\infty )$) either $|\xi | \geq |\eta | - |\xi - \eta | \geq 2^{j+2} - 2^{j-2} \geq 2^{j+1}$ or $|\xi | \leq |\eta | + |\xi -\eta | \leq 2^{j-2}+ 2^{j-2} = 2^{j-1}$. In any case $p_j (\xi )=0$ and so $P_j (\phi_1 (1-\widetilde{\widetilde{P}}_j)f )=0$. The part involving $\phi_2$ can be estimated by $e_d (j)$ using the same argument as above.
\end{proof}
\subsection{Cubes}\label{sec_cubes}
We denote by $Q$ any open cube in $\RR^3$. Given $a>1$ we denote by $aQ$ the cube with the same center as $Q$ and $a$ times larger sidelength. We sometimes write $Q(x)$ to emphasize that cube $Q$ is centered at a point $x\in \RR^3$.
Given an open cube $Q$ of sidelength $d>0$ we let $\phi_Q\in C_0^\infty (\RR^3 ; [0,1])$ be a $d$-cutoff such that 
\eqnb\label{what_is_phi_Q}
\phi_Q=1\text{ on }Q,\quad \supp \,\phi_Q \subset 7Q/6,\quad \text{ and }\quad \| \nabla^k \phi_Q \|_\infty \leq C_k d^{-k} .
\eqne
Note that  
\eqnb\label{bound_sup_on_derivs_of_phi_in_FS}
|\xi |^k |\widehat{\phi_{Q} } (\xi )| \leq c_k d^{3-k}\qquad \text{ for } \xi \in \RR^3,
\eqne
which can be shown by a direct computation.

\subsection{Localised Bernstein inequalities}
If $Q$ is a cube of sidelength $d>2^{-j}$ then
\eqnb\label{localised_bernstein_ineq_single}
\| \phi_Q  P_j f\|_q \leq c\, 2^{3j\left( \frac{1}{2} - \frac{1}{q} \right)} \| \phi_Q P_j f\| + e_d (j) \| f \|_q,
\eqne
due to Lemma \ref{lem_moving_lp_across_bumps} and the classical Bernstein inequality \eqref{bernstein_ineq_single}.

\subsection{Absolute continuity}
Here we state two lemmas that will help us (in Step~1 of the proof of Proposition~\ref{prop_basic_estimate}) in proving the main estimate for Leray--Hopf weak solutions.
\begin{lemma}\label{lem_abs_cont_1}
Suppose that $f\colon [a,b] \to \RR$ is continuous and such that $f'\in L^1 (a,b)$. Then
\[
f(t)=f(s) + \int_s^t f'(\tau ) \d \tau
\]
for every $s,t \in (a,b)$.
\end{lemma}
\begin{proof} This is elementary.
\end{proof}
\begin{lemma}
If $u(x,t)$ is weakly continuous in time on an interval $(a,b)$ with values in $L^2 (\RR^3)$ then $P_j u $ is strongly continuous in time into $L^2 (\Omega )$ on $(a,b)$ for any bounded domain $\Omega \subset \RR^3$
\end{lemma}
\begin{proof}
We note that
\[
\begin{split}
\| P_j u(t) - P_j u (s) \|_{L^2(\Omega )}^2 &= \int_{\Omega } \left| \int \check{p}_j (x-y) \left( u(y,t)-u(y,s) \right) \d y \right|^2 \d x.
\end{split}
\]
Weak continuity of $u(t)$ gives that the integral inside the absolute value converges to $0$ as $t\to s$ (for any fixed $x$). Furthermore it is bounded by
\[
\| \check{p}_j \| \, \| u (t) - u (s) \| \leq c_j ,
\]
where we used the Cauchy-Schwarz inequality and the fact that $u$ is bounded in $L^2$ (a property of functions weakly continuous in $L^2$). Since the constant function $c_j^2$ is integrable on $\Omega $, the claim of the lemma follows from the Dominated Convergence Theorem.
\end{proof}

\section{The proof of the main result}\label{sec_main_result}
In this section we prove Theorem \ref{thm_main}, namely we will show that $d_H(S)\leq 5-4\alpha $, where $S$ is the singular set in space of a Leray--Hopf weak solution (recall \eqref{def_singset}). We will actually show that 
\[d_H(S) \leq 5-4\alpha +\varepsilon \]
for any 
\eqnb\label{how_small_is_eps}
\varepsilon \in (0, \min ((4\alpha -4)/3,1/20) ).
\eqne
We now fix such $\varepsilon $ and we allow every constant (denoted by ``$c$'') to depend on $\varepsilon$.

We say that a cube $Q$ is a $j$-cube if it has sidelength $2^{-j(1-\varepsilon)}$. The reason for considering such ``almost dyadic cubes'' (rather than the dyadic cubes of sidelength $2^{-j}$) is that $e_d(j)=e(j)$ for $d=2^{-j(1-\varepsilon )}$ (which is not true for $d=2^{-j}$).
We say that a cover of a set is a $j-$\emph{cover} if it consists only of $j$-cubes. We denote by $S_j (\Omega )$ any $j$-cover of $\Omega$ such that $\# S_j (\Omega ) \leq c (\mathrm{diam} (\Omega ) /2^{-j(1-\varepsilon )} )^3$.

Moreover, given a $j$-cube and $k\in \ZZ$ we denote the $k$-cube cocentric with $Q$ by $Q_k$, that is
\[
Q_k \coloneqq 2^{(j-k) (1-\varepsilon )} Q.
\]
\subsection{The main estimate}
Given a cube $Q$ and $j\in \ZZ$ we let 
\[
u_{Q,j} \coloneqq \| \phi_Q P_j u \|
\]
and we write
\[
u_{Q,j\pm 2} \coloneqq \sum_{k=j-2}^{j+2} u_{Q,k }
\]
We point out that $u_{Q,j} $ is a function of time, which we will often skip in our notation.

We start with a derivation of an estimate for $u_{Q,j}$ for any $j\in \ZZ$ and any cube $Q$ of sidelength $d>16 \cdot 2^{-j}$. 
\begin{proposition}[Main estimate]\label{prop_basic_estimate}
Let $u$ be a Leray-Hopf weak solution of the hyperdissipative Navier--Stokes equations \eqref{NSE_intro} on time interval $[0,\infty )$ and let $d>16\cdot 2^{-j}$. Then $u_{Q,j}$ is continuous on $[0,\infty )$ and 
\eqnb\label{basic_est}\begin{split}
\frac{\d }{\d t} u_{Q,j}^2  &\leq -c\, 2^{2\alpha j} u_{Q,j}^2+c \,u_{Q,j} \left( 2^j u_{3Q/2,j\pm 2} \sum_{\theta j \leq k\leq j-5 } 2^{3k/2} u_{\max(Q_k,3Q/2),k}  + 2^{5j/2}  u_{3Q/2,j\pm 4}^2 \right. \\
&\hspace{5cm}\left.+ 2^{3j/2} \sum_{k\geq j+1} 2^{k} u_{3Q/2,k}^2  \right) + e_{diss} + \sum_{k\geq \theta j } e_d (k)  \\
&= - G_{diss}  + c \,u_{Q,j} \left( G_{low,loc} + G_{loc} + G_{hh}  \right) +e_{diss} +e_{vl}+ \sum_{k\geq \theta j } e_d (k) 
\end{split}
\eqne
is satisfied in the integral sense (recall \eqref{integral_sense}) for any cube $Q$ of side-length $d$ and any $j\in \ZZ$, where 
\eqnb\label{def_of_theta}
\theta \coloneqq 2(2\alpha -1-\varepsilon)/3
\eqne
and
\[\begin{split}
e_{diss} &\coloneqq c\,2^{2\alpha j} (d2^j )^{-1} u_{3Q/2,j\pm 2 }^2 , \\
e_{vl} &\coloneqq c\, 2^{2\alpha j} 2^{-\varepsilon j}u_{3Q/2,j\pm 2}^2.
\end{split}
\]
\end{proposition}
Here $\max (Q_k, 3Q/2)$ denotes the larger of the cubes $Q_k$, $3Q/2$, $G_{diss}$ should be thought of as the \emph{dissipation term}, $G_{low,loc}$ the interaction between low (i.e. modes $k\leq j-5$) and local modes (i.e. modes $j\pm 2$), $G_{loc}$ the local interactions (i.e. including only the modes $j\pm 4$) and $G_{hh}$ the interactions between high modes (i.e. modes $k\geq j$). 

The role of the parameter $\theta $ is to separate the ``very low'' Littlewood-projections from the ``low'' Littlewood-Paley projections. That is (roughly speaking), given $j\in \NN$ we will not have to worry about the Littlewood-Paley projections $P_k$ with $k< \theta j$ (i.e. they will be effortlessly absorbed by the dissipation at the price of the error term $e_{vl}$ (``vl'' here stands for ``very low''), see \eqref{temp1_low_vs_verylow}-\eqref{temp2_low_vs_verylow} below for a detailed explanation), which is the reason why such modes are not included in $G_{low,loc}$. In fact $G_{low,loc}$ is (roughly speaking) the most dangerous term, as it represents, in a sense, the injection of energy from low scales to high scales, and we will need to use $G_{diss}$ to counteract it, see Step~5 in the proof of Theorem \ref{thm_reg_outside_E}.

The error term $e_{diss}$ appearing in the estimate is the error appearing when estimating the dissipation term and it cannot be estimated by $e_d (j)$. Its appearance is a drawback of the main estimate, but in our applications (in Theorem \ref{thm_j_good_gives_critical regularity} and Theorem \ref{thm_reg_outside_E}) it can be absorbed by $G_{diss}$.
\begin{proof}(of Proposition \ref{prop_basic_estimate})
Recall \eqref{int_of_reg} that a Leray-Hopf weak solution admits intervals of regularity. 

\emph{Step~1.} We show that it is sufficient to show \eqref{basic_est} on each of the intervals of regularity.\\

On each interval of regularity $(a,b)$ we apply the Leray projection (recall \eqref{def_of_T} to the first equation of \eqref{NSE_intro} to obtain 
\[
u_t +(- \Delta )^\alpha  u + T[(u\cdot \nabla )u] =0.
\]
Multiplying by $P_j (\phi_Q^2 {P_ju})$ and integrating in space we obtain (at any given time)
\[
\frac{1}{2}\frac{\d }{\d t} u_{Q,j}^2 = -\int (-\Delta )^\alpha u \, P_j (\phi_Q^2 {P_ju}) - \int T[(u\cdot \nabla )u] \, P_j (\phi_Q^2 {P_ju}) =: I + J.
\]
We note that $I,J\in L^1 (0,T)$ for every $T>0$. Indeed, by brutal estimates
\[
\begin{split}
|J| &=  \left| \int \phi_Q P_j T[(u\cdot \nabla )u] \, \phi_Q^2 {P_ju} \right| \\
&\leq   \left\|   P_j T[(u\cdot \nabla )u]  \right\|_1 \| P_j u \|_\infty \\
&\leq c \| u\| \| \nabla u\| \cdot 2^{3j/2} \| P_j u \| \\
&\leq c \, 2^{3j/2} \| \nabla u\|
\end{split}
\]
(where we used Bernstein inequality \eqref{bernstein_ineq_single} in the third line), which is integrable on $(0,T)$ for every $T>0$. That $I\in L^1 (0,T)$ for every $T>0$ is a consequence of Step~2 below. Thus, since $u(t)$ is weakly continuous with values in $L^2$ (recall Section \ref{sec_prelims}), Lemma \ref{lem_abs_cont_1} gives that \eqref{basic_est} is valid (in the integral sense) on $[0,\infty )$. 

Thus it suffices to show that $I+J$ can be estimated by the right-hand side of \eqref{basic_est}.\vspace{0.4cm}\\
\emph{Step~2.} We show that $I\leq - G_{diss} + e_{diss}  + e_d (j)$.\\
(Note that this gives in particular that $I\in L^1 (0,\infty )$, since (trivially) $u_{Q',j}\leq c$ for every cube $Q'$ and every $j$.)\\

We write
\[\begin{split}
I &= - \int \phi_Q (-\Delta )^\alpha P_j u \, \phi_Q {P_ju} \\
&= - \int (-\Delta )^\alpha \widetilde{P}_j ( \phi_Q P_j u ) \, \phi_Q {P_ju} - \int (-\Delta )^\alpha (1- \widetilde{P}_j ) ( \phi_Q P_j u ) \, \phi_Q {P_ju}   -\int [\phi_Q , (-\Delta )^\alpha ] P_j u \, \phi_Q {P_ju}  \\
&=: I_1 + I_2 + I_3 
\end{split} \]
Note that, due to the Plancherel theorem
\[\begin{split}
I_1 &= - c \int | \xi |^{2\alpha } \widetilde{p_j} (\xi ) | \widehat{v} (\xi )|^2 \d \xi \\
&\leq - c \,2^{2\alpha j } \int\widetilde{p_j} (\xi ) | \widehat{v} (\xi ) |^2 \d \xi \\
& =   -c\, 2^{2\alpha j} \int \widetilde{P}_j v\cdot {v}\\
& = - c\, 2^{2\alpha j } u_{Q,j}^2 + c\, 2^{2\alpha j } \int (1-\widetilde{P}_j )v \cdot {v}\\
&\leq  - c\, 2^{2\alpha j } u_{Q,j}^2 + c\, 2^{2\alpha j } \| (1 -\widetilde{P}_j ) v \|  \\
&= - G_{diss} + e_d(j)
\end{split}\]
where we wrote $v\coloneqq \phi_Q P_j u$ for brevity, and we used the fact that $\| v \|\leq c$ (recall \eqref{EI_prelims}) in the last two lines as well as Lemma \ref{lem_moving_lp_across_bumps} in the last line. \vspace{0.4cm}\\
\emph{Step~2.1} We show that $I_2 \leq e_d(j)$.\\

We write
\[
I_2 \leq \| (-\Delta )^\alpha (1- \widetilde{P}_j ) ( \phi_Q {P_ju} ) \| u_{Q,j},
\]
and we will show that 
\eqnb\label{1.1_temp}
\| (-\Delta )^\alpha (1- \widetilde{P}_j ) ( \phi_Q {P_ju} ) \| \leq e_d (j).
\eqne
(This completes this step as $u_{Q,j}\leq c$, as above.) Indeed, \eqref{1.1_temp} follows in a similar way as Lemma \ref{lem_moving_lp_across_bumps} by decomposing
\[
\phi_Q = \phi_1 + \phi_2 ,
\] 
where 
\[\begin{split}
\widehat{\phi_1} (\xi )  &\coloneqq \chi_{|\xi |\leq 2^{j-2}} \widehat{ \phi_Q }(\xi ),\\
\widehat{\phi_2} (\xi ) & \coloneqq \chi_{|\xi |> 2^{j-2}} \widehat{ \phi_Q }(\xi ).
\end{split}\]
We see that $\phi_1 P_j u = \widetilde{P}_j (\phi_1 P_j u) $ (because of the supports in Fourier space, cf. \eqref{phi1_is_ok}) and so it is sufficient to show that 
\[
\| (-\Delta )^\alpha  ( \phi_2 P_j u ) \| \leq e_d (j)
\]
(since the operator norm $\| 1- \widetilde{P}_j \|\leq 1$). Since the Fourier transform of $(-\Delta )^{\alpha } ( \phi_2 P_j u )$ is
\[\begin{split}
c| \xi |^{2\alpha } \int  \widehat{\phi_2 } (\xi - \eta ) p_j (\eta ) \widehat{u} ( \eta ) \,\d \eta &\leq  c \int  |\xi - \eta |^{2\alpha } |  \widehat{\phi_2 } (\xi - \eta ) p_j (\eta ) \widehat{u} ( \eta ) | \d \eta\\
&\hspace{0.3cm}+c \int |\eta |^{2\alpha }   |  \widehat{\phi_2 } (\xi - \eta )  p_j (\eta ) \widehat{u} ( \eta ) | \d \eta
\end{split}
\]
we obtain
 \[
 \| (-\Delta )^\alpha  ( \phi_2 P_j u ) \| \leq c \| u \| \int_{|\xi |>2^{j-2}} |\xi |^{2\alpha } |\widehat{\phi_2} (\xi ) | \d \xi  + c \| \widehat{\phi_2 } \|_1 \| (-\Delta )^{2\alpha } P_j u \|  \leq e_d (j), 
 \]
where we used the Plancherel theorem, \eqref{will_show_this_1} and the fact that $\| \,|\cdot |^{2\alpha } \widehat{\phi_2 } (\cdot ) \|_1\leq e_d (j)$ (which follows in the same way as \eqref{will_show_this_1}).\vspace{0.4cm}\\
\emph{Step~2.2.} We show that $I_3\leq  e_{diss}  + e_d(j)$.\\

We have
\[
I_3 \leq \| [\phi_Q , (-\Delta )^{\alpha } ] P_j u \| u_{Q,j}
\]
For brevity we let $v\coloneqq P_j (\phi_{3Q/2} u )$, $\phi \coloneqq \phi_Q$ and 
\[
W \coloneqq [\phi ,(-\Delta )^\alpha ] v.
\]
We will show below that 
\[
\| W \| \leq   c\,2^{2\alpha j} (d\,2^j)^{-1} u_{3Q/2,j} + e_d(j),
\]
and we will show in Step~2.2c that
\eqnb\label{to_show_1.2c}
\| W \| = \| [\phi , (-\Delta )^{\alpha } ] P_j u \| + e_d(j),
\eqne
from which the claim of this step follows (and so, together with Step~2.1, finishes Step~2). Since
\[
\begin{split}
\widehat{W} (\xi ) &=c\int \left( |\eta |^{2\alpha } - |\xi |^{2\alpha } \right) \widehat{\phi } (\xi - \eta )  \widehat{v } (\eta ) \, \d \eta,
\end{split}
\]
we can decompose $W$ by writing $\int= \int_{|\eta - \xi |\leq 2^{j-3}} + \int_{|\eta - \xi |>2^{j-3}}$, that is 
\[W=W_1+W_2,\]
where
\[\begin{split}
\widehat{W_1} (\xi )&\coloneqq c\int_{|\eta - \xi |\leq 2^{j-3}} \left( |\eta |^{2\alpha } - |\xi |^{2\alpha } \right) \widehat{\phi } (\xi - \eta )  \widehat{v } (\eta ) \, \d \eta,\\
\widehat{W_2} (\xi )&\coloneqq c\int_{|\eta - \xi |>2^{j-3}} \left( |\eta |^{2\alpha } - |\xi |^{2\alpha } \right) \widehat{\phi } (\xi - \eta )  \widehat{v } (\eta ) \, \d \eta.
\end{split}\]
We will show (in Step~2.2b below) that $\| W_2 \|\leq e_d (j)$. As for $W_1$, note that, since $\mathrm{supp}\,{p_j }\subset \{ |\eta |\in (2^{j-1},2^{j+1}) \}$, 
\eqnb\label{supp_of_FT_of_Q1} \mathrm{supp}\,\widehat{W_1}\subset \{ |\xi |\in (  2^{j-2}, 2^{j+2} ) \} .\eqne
Setting $f(z)\coloneqq z^\alpha $ and expanding it in the Taylor series around $|\xi |^2$ we obtain
\[
|\eta |^{2\alpha } - |\xi |^{2\alpha } = \sum_{k=1}^3 \frac{ f^{(k)} (|\xi |^2) }{k !} \left( |\eta |^2 - |\xi |^2 \right)^k  +\frac{ f^{(4)} (z_0 ) }{24} \left( |\eta |^2 - |\xi |^2 \right)^4,
\]
where $z_0$ belongs to the interval with endpoints $|\eta |^2$ and $|\xi |^2$ (and so in particular $z_0\in [2^{2j-4},2^{2j+4}]$). Writing $ |\eta |^2 - |\xi |^2 = \sum_{i=1}^3 (\eta_i - \xi_i )(\eta_i + \xi_i )$ and taking $k$-th power we obtain
\[
|\eta |^{2\alpha } - |\xi |^{2\alpha } = \sum_{k=1}^4 c_k f^{(k)} (z)  \sum_{|\beta |=k, |\gamma_1 |+ |\gamma_2 |=k } c_{ \beta \gamma_1\gamma_2} (\eta - \xi )^\beta \eta^{\gamma_1} \xi^{\gamma_2},
\]
where $z=|\xi |^2$ (for $k\leq 3$) or $z=z_0$ (for $k=4$).
Thus, noting that $| f^{(k)} (z) |\leq c \,2^{j(2\alpha -2k)}$,
\[\begin{split}
|\widehat{W_1} (\xi )| &\leq c  \sum_{k=1}^3 \left|   f^{(k)} (|\xi |^2 )    \right|\sum_{|\beta |=k, |\gamma_1 |+ |\gamma_2 |=k } |\xi |^{|\gamma_2 |} \left| \int_{|\eta - \xi |\leq 2^{j-3}} ( \xi - \eta )^\beta  \widehat{\phi } (\xi - \eta ) \eta^{\gamma_1} \widehat{v } (\eta )    \, \d \eta \right| \\
&+ c\sum_{|\beta |=4, |\gamma_1 |+ |\gamma_2 |=4 } |\xi |^{|\gamma_2 |} \left| \int_{|\eta - \xi |\leq 2^{j-3}} f^{(k)} (z_0) ( \xi - \eta )^\beta  \widehat{\phi } (\xi - \eta ) \eta^{\gamma_1} \widehat{v } (\eta )    \, \d \eta \right|\\
&\leq c  \sum_{k=1}^3 \sum_{|\beta |=k, |\gamma_1 |+ |\gamma_2 |=k } 2^{j(2\alpha -2k +|\gamma_2 | )}  \left| \int_{|\eta - \xi |\leq 2^{j-3}} ( \xi - \eta )^\beta  \widehat{\phi } (\xi - \eta ) \eta^{\gamma_1} \widehat{v } (\eta )    \, \d \eta \right| \\
&+ c \,2^{j(2\alpha - 4 )} \int_{|\eta - \xi |\leq 2^{j-3}}    | \xi - \eta |^4 \left|  \widehat{\phi } (\xi - \eta )  \widehat{v } (\eta )     \right| \d \eta \\
&\leq c  \sum_{k=1}^3 \sum_{|\beta |=k, |\gamma_1 |+ |\gamma_2 |=k } 2^{j(2\alpha -2k +|\gamma_2 | )}  \left| \widehat{ D^\beta \phi D^{\gamma_1} v } (\xi )  \right|  \\
&+ c  \sum_{k=1}^3 \sum_{|\beta |=k, |\gamma_1 |+ |\gamma_2 |=k } 2^{j(2\alpha -2k +|\gamma_2 | )}  \left| \int_{|\eta - \xi |> 2^{j-3}} ( \xi - \eta )^\beta  \widehat{\phi } (\xi - \eta ) \eta^{\gamma_1} \widehat{v } (\eta )    \, \d \eta \right| \\
&+ c \,2^{j(2\alpha - 4 )} \int_{|\eta - \xi |\leq 2^{j-3}}    | \xi - \eta |^4 \left|  \widehat{\phi } (\xi - \eta )  \widehat{v } (\eta )     \right| \d \eta\\
&=: c  \sum_{k=1}^3 \sum_{|\beta |=k, |\gamma_1 |+ |\gamma_2 |=k } 2^{j(2\alpha -2k +|\gamma_2 | )}  \left| \widehat{ D^\beta  \phi D^{\gamma_1} v } (\xi )  \right|  + \mathrm{Err}_1 (\xi )+ \mathrm{Err}_2 (\xi ).
\end{split}
\]
We will show below (in Step~2.2a below) that 
\[
\| \mathrm{Err}_1 \|, \| \mathrm{Err}_2 \|  \leq c 2^{2\alpha j } (d\, 2^j  )^{-1} u_{3Q/2, j\pm 2}  + e_d (j) .
\]
This, together with the Plancherel identity gives
\[\begin{split}
\| W_1 \| &\leq c  \sum_{k=1}^3 \sum_{|\beta  |=k, |\gamma_1 |+ |\gamma_2 |=k } 2^{j(2\alpha -2k +|\gamma_2 | )}  \|  D^\beta \phi D^{\gamma_1} v   \| + c 2^{2\alpha j } (d\, 2^j  )^{-1} u_{3Q/2, j\pm 2}  + e_d (j) \\
&\leq  c  \sum_{k=1}^3 2^{2\alpha j }  (d\,2^j )^{-k}\|    v   \|  +    c 2^{2\alpha j } (d\, 2^j  )^{-1} u_{3Q/2, j\pm 2}  + e_d (j)  
\end{split}
\]
where we used the facts that $|\nabla^k \phi |\leq c\, d^{-k}$ for $k=1,2,3$, and $\| D^{\gamma_1} v \| \leq c\,2^{j|\gamma_2 |} \| v\|$ (by applying Lemma \ref{lem_moving_lp_across_bumps}). Since $d>2^{-j}$ and $\| v \| \leq  \| \phi_{3Q/2} \widetilde{P_j} u \| + e_d(j) = u_{3Q/2,j\pm 2 } + e_d(j)$ (where we applied Corollary \ref{cor_put_Pj_inside}) we thus arrive at 
\[
\| W_1 \| \leq c 2^{2\alpha j } (d\, 2^j  )^{-1} u_{3Q/2, j\pm 2}  + e_d (j)  ,
\]
as required.\vspace{0.4cm}\\
\emph{Step~2.2a} We show that $\| \mathrm{Err}_1 \| \leq e_d (j)$ and $\| \mathrm{Err}_2 \| \leq c 2^{2\alpha j } (d\, 2^j  )^{-1} u_{3Q/2, j\pm 2} + e_d (j)$.\\

We focus on $\mathrm{Err}_1$ first. We have
\[\begin{split}
\mathrm{Err}_1 (\xi ) &= c  \sum_{k=1}^3 \sum_{|\beta |=k, |\gamma_1 |+ |\gamma_2 |=k } 2^{j(2\alpha -2k +|\gamma_2 | )}  \left| \int_{|\eta - \xi |> 2^{j-3}} ( \xi - \eta )^\beta  \widehat{\phi } (\xi - \eta ) \eta^{\gamma_1} \widehat{v } (\eta )    \, \d \eta \right|\\
&\leq c  \sum_{k=1}^3 2^{j(2\alpha -k )}   \int_{|\eta - \xi |> 2^{j-3}} | \xi - \eta |^k   \left| \widehat{\phi } (\xi - \eta )  \widehat{v } (\eta )      \right| \d \eta\\
&\leq c  2^{j(2\alpha -K )}   \int_{|\eta - \xi |> 2^{j-3}} | \xi - \eta |^K   \left| \widehat{\phi } (\xi - \eta )  \widehat{v } (\eta )      \right| \d \eta\\
&\leq  c_K \,2^{j(2\alpha-K )} d^{1-K}\int_{|\eta - \xi |>2^j } |\xi - \eta |^{-2} \left| \widehat{v} (\eta ) \right| \d \eta\\
&\leq c_K \,2^{j(2\alpha -1)} (d \,2^j)^{(1-K)} \left( \int_{|\eta - \xi |>2^{j-5} } | \xi - \eta |^{-4} \d \eta \right)^{1/2}   
\end{split}
\]
for every $K>3$, where we used \eqref{bound_sup_on_derivs_of_phi_in_FS} in the fourth line as well as the Cauchy-Schwarz inequality, \eqref{def_of_pj} and the fact that $\|v\|\leq \| u \|\leq c$ (recall \eqref{EI_prelims}) in the last line. Thus $\mathrm{Err}_1 (\xi )\leq  e_d (j)$ for every $\xi \in \RR^3$, and hence (since $|\xi |\leq 2^{j+2}$) also $\| \mathrm{Err}_1 \| \leq e_d (j)$.

As for $\mathrm{Err}_2$ we write 
\[\begin{split}
\mathrm{Err}_2 (\xi ) &= c\, 2^{j(2\alpha -4) }\int_{|\eta - \xi |\leq 2^{j-3}} |\xi - \eta |^4  \left| \widehat{\phi } (\xi - \eta )  \widehat{v } (\eta ) \right|   \d \eta\\
&\leq c\, 2^{j(2\alpha -4) } d^{-1} \int_{|\eta - \xi |\leq 2^{j-3}}   \left|   \widehat{v } (\eta ) \right| \d \eta\\
&\leq c\, 2^{j(2\alpha -3/2)} (d \,2^j )^{-1} \| v \|\\
&= c\, 2^{j(2\alpha -3/2)} (d \,2^j )^{-1} \| P_j \phi_{3Q/2} u \|\\
&\leq c\, 2^{j(2\alpha -3/2)} (d \,2^j )^{-1} u_{3Q/2,j\pm 2} + e_d (j)
\end{split}
\]
where we used \eqref{bound_sup_on_derivs_of_phi_in_FS} in the second line, the Cauchy-Schwarz inequality (as above) in the third line, and Corollary \ref{cor_put_Pj_inside} in the last line. Thus
\[
\| \mathrm{Err}_2 \| \leq c\, 2^{2\alpha j} (d\,2^j )^{-1} u_{3Q/2,j\pm 2} + e_d (j),
\]
as required.\vspace{0.4cm}\\
\emph{Step~2.2b} We show that $\| W_2 \| \leq e_d(j)$.\\

Indeed, since $|\xi |^{2\alpha } \leq c |\eta |^{2\alpha } + c |\xi - \eta |^{2\alpha}$ we obtain for any $K>2\alpha $
\[\begin{split}
\left| \widehat{W_2} (\xi ) \right| &= \left| \int_{|\eta - \xi |>2^{j-5}} \left( |\eta |^{2\alpha } - |\xi |^{2\alpha } \right) \widehat{\phi } (\xi - \eta ) \widehat{v } (\eta ) \, \d \eta\right| \\
&\leq  c \int_{|\eta - \xi |>2^{j-5}} |\eta |^{2\alpha }  \left| \widehat{\phi } (\xi - \eta )  \widehat{v }  (\eta )\right|  \d \eta + c \int_{|\eta - \xi |>2^{j-5}} |\xi -\eta |^{2\alpha }  \left| \widehat{\phi } (\xi - \eta )  \widehat{v }  (\eta ) \right|  \d \eta  \\
&\leq c_K 2^{ j(2\alpha - K)} \int_{|\eta - \xi |>2^{j-5}} |\xi - \eta |^{K }  \left| \widehat{\phi } (\xi - \eta )  \widehat{v }  (\eta )\right| \d \eta ,
\end{split}
\]
where we used the inequality $1< c_K | \xi - \eta  |^K 2^{-jK}$ as well as $|\eta |\leq c 2^{j}$ inside the first integral in the second line and the inequality $1\leq c_K |\xi - \eta |^{K-2\alpha } 2^{-j(K-2\alpha )}$ inside the second integral. Thus, using the Plancherel identity and Young's inequality for convolutions
\[\begin{split}
\| W_2 \| &= \| \widehat{ W_2} \| \\
&\leq c_K 2^{j(2\alpha - K)} \| v \| \int_{|\eta |> 2^{j-5}} |\eta |^K \left| \widehat{\phi } (\eta ) \right| \, \d \eta \\
&\leq c_K 2^{j(2\alpha - K)} \int_{|\eta |> 2^{j-5}} |\eta |^{K+4} \left| \widehat{\phi } (\eta ) \right| \, |\eta |^{-4}\, \d \eta \\
&\leq c_K 2^{j(2\alpha - K)} d^{-(K+1)} \int_{|\eta |> 2^{j-5}}  |\eta |^{-4}\, \d \eta \\
&= c_K 2^{2\alpha j } (d \,2^j )^{-(K+1)} ,
\end{split}
\]
as required, where we used \eqref{bound_sup_on_derivs_of_phi_in_FS} in the third inequality.\vspace{0.4cm}\\
\emph{Step~2.2c} We show that $\| [\phi , (-\Delta )^\alpha ] P_j (1-\phi_{3Q/2} )u  \| \leq e_d (j)$. (This implies \eqref{to_show_1.2c}.)\\

Indeed, letting (for brevity) $w\coloneqq (1- \phi_{3Q/2} ) u$ and $q_j (\xi ) \coloneqq | \xi |^{2\alpha } p_j (\xi )$ we can write 
\[
\phi (-\Delta )^\alpha P_j w (x) = \phi (x) \int_{\{ |x-y | \geq d/3 \} } \check{q_j} (x-y )w(y) \, \d y,
\]
as in \eqref{moving_bumps_temp}. Thus, since $\| \check{q_j} \|_{L^1 (B(d/3)^c) } \leq e_d (j)  $ (as in \eqref{Lq_norm_of_pj_outside_delta_ball}) we can use Young's inequality for convolutions to obtain
\eqnb\label{1.2c_temp}
\| \phi (-\Delta )^\alpha P_j w  \|  \leq \| \check{q_j} \|_{L^1 (B(d/3)^c) } \| w \|  \leq e_d(j).
\eqne
On the other hand 
\[
\begin{split}
\left\|  (-\Delta )^\alpha \left( \phi P_j w \right) \right\|& \leq  \left\|  (-\Delta )^\alpha \widetilde{P_j} \left( \phi P_j w \right) \right\| + \left\|  (-\Delta )^\alpha (1-\widetilde{P_j} )\left( \phi P_j w \right) \right\| \\
&\leq c 2^{2\alpha j} \| \phi P_j w \| +  e_d (j) \\
&\leq e_d (j),
\end{split}
\]
where we used \eqref{1.1_temp} (applied with $w$ instead of $u$) in the second line and Lemma \ref{lem_moving_bump_functions_across_LP} in the last line. This and \eqref{1.2c_temp} prove the claim. \vspace{0.4cm}\\
\emph{Step~3.} We show that $J \leq c \,u_{Q,j} \left( G_{low,loc} + G_{loc} + G_{hh}  \right) + e_{vl} + \sum_{k\geq \theta j}e_d(k)$.\\(This together with Step~2 finishes the proof.)

We can rewrite $J $ in the form 
\[\begin{split}
J &= -\int \phi_Q P_j T[(u\cdot \nabla )u] \cdot (\phi_Q P_ju) \\
& = - \sum_{i,l,m}\int \phi_Q  T_{m i} P_j  (u_l \,\p_l u_m ) \phi_Q P_j u_i
\end{split}\]
where we used the fact that ``$T_{m i}$'' and  ``$P_j$'' are multipliers (so that they commute). (Recall that $\widehat{T_{m i}}(\xi) = (\delta_{mi} - \xi_m \xi_i |\xi |^{-2} )$, see \eqref{def_of_T}.) We now apply the paraproduct formula \eqref{bony} to $P_j (u_l\p_l u_m)$ to write
\[
J=  J_{loc,low} +J_{low,loc} +J_{loc} + J_{hh}  ,
\]
where each of $J_{loc,low}$, $J_{low,loc}$, $J_{loc}$, $J_{hh}$ equals $J$ except for the term $u_l \p_l u_m$, which is replaced by the corresponding combination of the modes of $u_l$ and $\p_l u_m$, as in the paraproduct formula (see \eqref{calculating_Jhh} and \eqref{bounding_lowloc} below). We estimate $J_{hh}$ in Step~3.1 below and $J_{loc,low}$, $J_{low,loc}$, $J_{loc}$ in Step~3.2.\vspace{0.4cm}\\
\emph{Step~3.1} We show that $J_{hh}\leq c \,u_{Q,j} G_{hh}+ \sum_{k\geq j} e_d (k)$.\\

We write 
\eqnb\label{calculating_Jhh}\begin{split}
J_{hh} &= -\sum_{i,l,m} \int \phi_Q  T_{m i} P_j \left( \sum_{k\geq j+3 } P_k u_l \widetilde{P}_{k}\p_l u_m \right) \phi_Q P_j u_i\\
& \leq \| \phi_Q P_j u \|_\infty   \sum_{i,l,m} \left\|  \phi_Q  T_{m i} P_j \sum_{k\geq j+3} \left(  P_k u_l \widetilde{P}_{k}\p_l u_m \right) \right\|_1 \\
& \leq c\, 2^{3j/2} u_{Q,j} \sum_{i,l,m} \left\|  \phi_Q  T_{m i} P_j \sum_{k\geq j+3} \left(  P_k u_l \widetilde{P}_{k}\p_l u_m \right) \right\|_1  + e_d (j) \\
& \leq c\, 2^{3j/2} u_{Q,j} \sum_{i,l,m} \left\|  \phi_Q  T_{m i} P_j \phi_{3Q/2}^3 \left( \sum_{k\geq j+3} P_k  u_l \widetilde{P}_{k} \p_l u_m  \right) \right\|_1  + e_d (j) \\
&\leq c\, 2^{3j/2} u_{Q,j} \sum_{k\geq j+3} \| \phi_{3Q/2} P_k  u \| \| \phi_{3Q/2}^2 \widetilde{P}_k \nabla u \|   + e_d (j),
\end{split}
\eqne
where, in the fourth line we applied Corollary \ref{cor_moving_bump_functions_across_LP+any_homog_mult} with $f\coloneqq \sum_{k\geq j+3} P_k u_l \widetilde{P}_k u_m $ and noted that $\supp \,\phi \subset 7Q/3$ is separated from $\supp \, (1-\phi_{3Q/2}^3 )$ by at least $d/3$. As for the  third line, we used the fact that $P_j u = P_j \widetilde{P}_j u$, \eqref{localised_bernstein_ineq_single} and \eqref{bernstein_ineq_single} to write 
$\| \phi_Q P_j u \|_\infty  \leq c 2^{3j/2 } u_{Q,j}+ e_d (j) \| \widetilde{P}_j u \|_\infty\leq c 2^{3j/2 } u_{Q,j}+ e_d (j)  $, as well as noted that $e_d (j)$ multiplied by the (long) $L^1$ norm still gives $e_d(j)$, since we can brutally estimate this norm,
\[\begin{split}
\left\|  \phi_Q  T_{m i} P_j \sum_{k\geq j+3} \left(  P_k u_l \widetilde{P}_{k}\p_l u_m \right) \right\|_1 &\leq \| \phi_Q \|\,  \left\| P_j   \p_l T_{m i}  \sum_{k\geq j+3} \left(  P_k u_l \widetilde{P}_{k} u_m \right) \right\| \\
&\leq c\, d^{3/2} 2^j \left\| P_j   T_{m i}  \sum_{k\geq j+3} \left(  P_k u_l \widetilde{P}_{k} u_m \right) \right\| \\
&\leq c\, d^{3/2} 2^{5j/2} \left\| P_j     \sum_{k\geq j+3} \left(  P_k u_l \widetilde{P}_{k} u_m \right) \right\|_1 \\
&\leq c\, d^{3/2} 2^{5j/2} \sum_{k\geq j+3} \left\|     P_k u_l \widetilde{P}_{k} u_m  \right\|_1 \\
&\leq c\, d^{3/2} 2^{5j/2} \sum_{k\geq j+1} \left\|     P_k u \right\|^2 \\
&\leq c\, d^{3/2} 2^{5j/2} \left\|      u \right\|^2 \\
&\leq c\, d^{3/2} 2^{5j/2} 
\end{split}
\]
for each $i,l,m$, where we used the Cauchy-Schwarz inequality in the first line, boundedness (in $L^2$) of the Leray projection (i.e. the fact that $|\widehat{T_{mi}}(\xi )| \leq 1$) and the Bernstein inequality \eqref{bernstein_ineq_single} in the third line, \eqref{lp_pj_is_bdd_on_Lp} in the fourth line and the Cauchy-Schwarz inequality (twice) in the fifth line.

Noting that
\[\begin{split} 
\| \phi_{3Q/2}^2  \widetilde{P}_k \nabla u \|  & = \|{P}_{k\pm 2} ( \phi_{3Q/2}^2 \nabla \widetilde{P}_k u ) \| + e_d(k)  \\
& \leq  \|{P}_{k\pm 2} \nabla ( \phi_{3Q/2}^2  \widetilde{P}_k u ) \|  + 2 \| {P}_{k\pm 2} ( \nabla \phi_{3Q/2} \, \phi_{3Q/2}  \widetilde{P}_k u )  \| + e_d(k) \\
&\leq c \, 2^k \| \phi_{3Q/2}^2 \widetilde{P}_k u \| + c \, d^{-1} u_{3Q/2,k\pm 2 } + e_d(k)\\
&\leq c \, 2^k  u_{3Q/2,k\pm 2 } + e_d(k),
\end{split} \]
where we used Lemma \ref{lem_moving_lp_across_bumps} in the first inequality, the fact that $\| \widetilde{P}_k \|\leq 1$ and \eqref{what_is_phi_Q} in the third inequality, and the assumption $d  > 2^{-j } > 2^{-k}$ in the last inequality, we obtain 
\eqnb\label{estimate_on_Jhh}
J_{hh} \leq c\,2^{3j/2} u_{Q,j} \sum_{k\geq j+1} 2^k u_{3Q/2,k}^2    + \sum_{k\geq j} e_d (k) ,
\eqne
as required, where we also applied the Cauchy-Schwarz inequality in the first sum.\vspace{0.4cm}\\
\emph{Step~3.2} We show that $J_{loc,low}+J_{low,loc}+J_{loc} \leq c\,u_{Q,j} \left( G_{low,loc}+G_{loc} \right) + e_{vl} + \sum_{k\geq \theta j } e_d (k)$. (This completes the proof of Step~3.)\\

We set 
\[
U_{lm} \coloneqq \widetilde{P}_j u_l \sum_{k\leq j-5} P_k u_m +  \widetilde{P}_j u_m \sum_{k\leq j-5} P_k u_l +  \left( \sum_{k=j-4}^{j+2} P_k u_l  \right) \left( \sum_{k=j-4}^{j+4} P_k u_m  \right)
\]
to write
\eqnb\label{bounding_lowloc}
\begin{split}
J_{loc,low}+J_{low,loc}+J_{loc} &= - \sum_{i,l,m} \int \phi_Q T_{mi} P_j \p_l U_{ml} \phi_Q P_j u_i\\
&\leq u_{Q,j}\sum_{i,l,m}  \|   \phi_Q T_{mi} P_j \p_l U_{ml} \|\\
&= u_{Q,j}\sum_{i,l,m} \|   \phi_Q T_{mi} P_j ( \phi_{3Q/2}^3 \p_l U_{ml} )\| + e_d (j) \\
&\leq c\, u_{Q,j} \sum_{l,m} \|  P_j ( \phi_{3Q/2}^3 \p_l U_{ml} ) \| + e_d (j) \\
&\leq c\, u_{Q,j} \sum_{l,m} \left( \|  P_j \p_l ( \phi_{3Q/2}^3  U_{ml} ) \| +  3\|  P_j ( \phi_{3Q/2}^2 \p_l  \phi_{3Q/2}  U_{ml} ) \| \right) + e_d (j) \\
&\leq c\,  2^j u_{Q,j} \sum_{l,m} \| \phi_{3Q/2}^2 U_{ml} \| + e_d (j),
\end{split}
\eqne
where we applied Corollary \ref{cor_moving_bump_functions_across_LP+any_homog_mult} (with $q\coloneqq 2$ and $f\coloneqq U_{ml}$) in the third line, as well as \eqref{what_is_phi_Q} (as in the previous calculation) and the assumption $d>2^{-j}$ in the last line.

We note that for each $m,l$
\eqnb\label{bound_phi_Uml}
 \| \phi_{3Q/2}^2 U_{ml} \| \leq 2 u_{3Q/2,j\pm 2} \left\| \phi_{3Q/2} \sum_{k\leq j-5 }   P_k u \right\|_\infty + \| \phi_{3Q/2} P_{j\pm 4 } u \|_\infty u_{3Q/2,j\pm 4} .
 \eqne
 Since we can estimate the above $L^\infty $ norm including the summation by writing $\sum_{k\leq j-5} = \sum_{k< \theta j } + \sum_{ \theta j\leq k \leq j-5}$,
\eqnb\label{temp1_low_vs_verylow}
\begin{split}
\left\| \phi_{3Q/2} \sum_{k\leq j-5 }   P_k u \right\|_\infty &\leq \left\| \phi_{3Q/2} \sum_{k< \theta  j }   P_k u \right\|_\infty + \sum_{ \theta j\leq k \leq j-5 } \left\| \phi_{\max(Q_k,3Q/2)}   P_k u \right\|_\infty \\
& \leq \| P_{\leq \theta j} u \|_\infty  + c\sum_{ \theta j\leq k \leq j-5 } 2^{3k/2} u_{\max(Q_k,3Q/2),k}  + \sum_{k\geq  \theta j } e_{d} (k)\\
&\leq c \,2^{3\theta j/2} + c\sum_{ \theta j\leq k \leq j-5 } 2^{3k/2} u_{\max(Q_k,3Q/2),k}  +  \sum_{k\geq  \theta j } e_{d} (k),
\end{split}\eqne
where we used the localised Bernstein inequality \eqref{localised_bernstein_ineq_single} in the second line (note that taking $\max(Q_k, 3Q/2 )$ is necessary since only then we can guarantee that the sidelength of such cube is greater than $ 2^{-k}$, as required by \eqref{localised_bernstein_ineq_single}) and the Bernstein inequality \eqref{bernstein_ineq_leq} in the last line, we can plug it in \eqref{bound_phi_Uml} to get
\[
 \| \phi_{3Q/2}^2 U_{ml} \| \leq c \, u_{3Q/2,j\pm 2 } 2^{3\theta j/2 } + c\, u_{3Q/2,j\pm 2 } \sum_{ \theta j\leq k \leq j-5 } 2^{3k/2} u_{3Q/2,k} + c\, 2^{3j/2 } u_{3Q/2,j\pm 4}^2 + \sum_{k\geq  \theta j } e_{d} (k),
\]
where we used the assumption $d>2^{-j+4}$ to apply the localised Bernstein inequality\eqref{localised_bernstein_ineq_single} again.
Inserting this into \eqref{bounding_lowloc} and using the fact that  $3\theta /2  = 2\alpha -1 -\varepsilon $ we obtain
\eqnb\label{temp2_low_vs_verylow}
\begin{split}
J_{loc,low}+J_{low,loc}+J_{loc} &\leq c\, 2^{2\alpha j} 2^{-\varepsilon j}u_{3Q/2,j\pm 2}^2 + c\,2^j u_{Q,j} u_{3Q/2,j\pm 2} \sum_{ \theta j\leq k \leq j-5 } 2^{3k/2} u_{3Q/2,k} \\
&+ c\, 2^{5j/2 } u_{Q,j} u_{3Q/2,j\pm 4}^2 + 
\sum_{k\geq  \theta j } e_{d} (k),
\end{split}
\eqne
as required (note the first term on the right-hand side the is the ``very low modes error'', $e_{vl}$).
\end{proof}
We now constraint ourselves to $j$-cubes. Given a $j$-cube $Q$ we will write
\[
u_Q \coloneqq u_{Q,j}
\]
for brevity. The above proposition then reduces to the following.
\begin{corollary}\label{cor_basic_est_for_jcubes}
Let $u$ be a Leray-Hopf weak solution of the Navier--Stokes equations \eqref{NSE_intro} on time interval $[0,\infty )$. Let $Q$ be a $j$-cube with $j$ large enough so that $2^{\varepsilon j } \geq 16$. Then
\eqnb\label{basic_estimate_jcubes}
\begin{split}
\frac{\d }{\d t} u_{Q}^2  &\leq -c\, 2^{2\alpha j} u_{Q}^2+c \,u_{Q} \left(  u_{3Q/2,j\pm 2} \sum_{\theta  j \leq k\leq j-5 } 2^{j+3k/2} u_{Q_k}  + 2^{5j/2}  u_{3Q/2,j\pm 4}^2 \right. \\
&\hspace{2cm}\left.+  \sum_{k\geq j+1} 2^{3j/2+k} u_{3Q/2,k}^2  \right) + c\, 2^{j(2\alpha -\varepsilon )} u_{3Q/2,j\pm 2}^2 + e(j)  
\end{split}
\eqne
\end{corollary}
\begin{proof}
We apply the estimate from Proposition \ref{prop_basic_estimate} (which is valid due to the assumption $2^{\varepsilon j} >16$). Since $e_{diss}\leq c\,2^{j(2\alpha -\varepsilon )}  u_{3Q/2,j\pm 2}^2$ and $\sum_{k\geq \theta j} e_d (k) \leq c_K \sum_{k\geq \theta j} 2^{ck} 2^{\varepsilon k (c-K)} \leq c_K 2^{c\theta j+\varepsilon \theta j (c-K)} = e(j)$, where $K$ is taken large enough (to guarantee the summability of the geometric series), we arrive at \eqref{basic_estimate_jcubes}, as required.
\end{proof}
\subsection{Good cubes and bad cubes}
We now fix $u_0\in H^1 (\RR^3)$ and a Leray-Hopf weak solution with initial data $u_0$. We say that a cube $Q$ is $j$-good if
\eqnb\label{j-good}
\int_0^\infty \int_Q \sum_{k\geq j } 2^{2\alpha k} | P_k u |^2   \leq  2^{-j(5-4\alpha +\varepsilon )}
\eqne
We say that a $j$-cube is good if it is $j$-good. Otherwise we say that it is bad. 
\subsection{Critical regularity on cubes with some good ancestors}

We show that, for sufficiently large $j$, goodness of a $j$-cube and some of its ancestors guarantees critical regularity ($+ \varepsilon $) of $u_Q$ on a smaller cube $Q$.

\begin{theorem}\label{thm_j_good_gives_critical regularity}
There exists $j_0>0$ (sufficiently large) such that whenever $Q$ is a $j$-cube with $j\geq j_0$ and such that each of $Q_{k-10}$, $k\in [ \theta j , j ]$, is good then
\[
u_{Q} (t) <  2^{-\frac{j}{2} (5- 4\alpha +\varepsilon )}\qquad \text{ for } t\in [0,T) .
\]
\end{theorem}
\begin{remark}\label{rem2}
The above theorem appears in an imprecise form as Theorem 7.1 in \cite{katz_pavlovic}\footnote{The claim following ``we must have'' on p. 374 does not follow, as the assumption of the proof by contradiction is only on $Q$, rather than on every cube in its nuclear family.}.  This is related to the somewhat unexpected way in which the dissipation error is handled by \cite{katz_pavlovic} in Lemma 6.3. This lemma is in fact not needed, and it seems necessary to incorporate the dissipation error directly into the main estimate (in order to get around the imprecision), as in $e_{diss}$ in \eqref{basic_est}.\\
Moreover the statement of Theorem 7.1 in \cite{katz_pavlovic} suggests that goodness of only one cube is sufficient for the critical decay, which is not consistent with its proof (which uses goodness of the ancestors in the third line on p. 375).
\end{remark}

\begin{proof}
Note that the claim is true for sufficiently small $t>0$ since $u_0\in H^1$ (so that $\| P_j u_0 \|^2 = \int p_j^2 (\xi )|\widehat{u_0} (\xi )|^2 \d \xi \leq c\, 2^{-2j} \int |\xi |^2 |\widehat{u_0} (\xi )|^2 \d \xi \leq c\,2^{-2j} \| u_0 \|_{H^1}^2 < 2^{-j(5-4\alpha +\varepsilon )}$ for sufficiently large $j$) and $u(t)$ remains bounded in $H^1$ for small $t>0$. Suppose that the theorem is false, and let $t_0$ be the first time when it fails and $Q$ a $j$-cube for which it fails. Then
\eqnb\label{uQj_is_crit+eps_until_t0}
u_{Q} (t) \leq   2^{-\frac{j}{2} (5- 4\alpha +\varepsilon )}\qquad \text{ for } t\leq t_0
\eqne
with equality for $t=t_0$.
Let $t_1\in (0,t_0)$ be the last time when $
u_{Q} (t_1) \leq \frac{1}{2}\,  2^{-\frac{j}{2} (5- 4\alpha +\varepsilon )}$, so that
\eqnb\label{temp_uQ_between_t1_t0}
\frac12 2^{-\frac{j}{2} (5- 4\alpha +\varepsilon )}\leq u_{Q} (t) \leq  2^{-\frac{j}{2} (5- 4\alpha +\varepsilon )}\qquad \text{ for }t\in (t_1,t_0).
\eqne
Note that, since $\mathrm{supp}\,\phi_{3Q/2} \subset 7Q/4 \subset Q_{j-1} \subset Q_{j-10}$ and $Q_{j-10}$ is good,
\[
\int_{t_1}^{t_0} \sum_{k \geq j-10 } 2^{2\alpha k} u_{3Q/2,k}^2  \leq c \int_{t_1}^{t_0} \int_{Q_{j-10}} \sum_{k \geq j-10 } 2^{2\alpha k } | P_k u |^2 \leq c\, 2^{-j (5-4\alpha + \varepsilon )},
\]
and so in particular (recalling that $\alpha \in (1, 5/4)$)
\eqnb\label{smart1}
\int_{t_1}^{t_0} u_{3Q/2,j\pm 4}^2 \leq c\, 2^{-j(5-2\alpha +\varepsilon )}
\eqne
and
\eqnb\label{smart2}
\int_{t_1}^{t_0} \sum_{k\geq j+1 } 2^{ k} u_{3Q/2,k}^2   \leq  c 2^{j (1-2\alpha  )}\int_{t_1}^{t_0} \sum_{k\geq j } 2^{2\alpha k} u_{3Q/2,k}^2  \leq  c 2^{-j (4-2\alpha + \varepsilon )}.
\eqne
Moreover, since $Q_{k-10}$ is good for every $k\in [ \theta j , j]$ we also have $\int_{t_1}^{t_0}  u_{Q_k}^2 \leq c 2^{-k (5-2\alpha +\varepsilon )}$ (as in \eqref{smart1}) and so
\eqnb\label{smart3}
\int_{t_1}^{t_0}  \sum_{\theta j \leq k\leq j-5 } 2^{3 k} u_{Q_k}^2  \leq c   \sum_{\theta j \leq k\leq j-5 }  2^{-k (2-2\alpha +\varepsilon )}  \leq  c 2^{-j (2-2\alpha + \varepsilon )},
\eqne
where we used the fact that $\alpha >1 $ and the fact that $\varepsilon >0$ is small (recall \eqref{how_small_is_eps}).\footnote{The restriction $\alpha >1$ is used here, but $\alpha \geq 1$ would be sufficient by noting that $\sum_{k\geq \theta j}2^{-k\varepsilon } \leq c\, 2^{-j\theta \varepsilon }$. Indeed, since $\theta >5/8$ (recall \eqref{def_of_theta}), the last inequality of this proof would then become $1\leq c\,2^{-j\varepsilon (\theta -1/2-1/8)}$, which still gives contradiction for large $j$.}

Applying the main estimate \eqref{basic_estimate_jcubes} between $t_1$ and $t_0$ (and ignoring the first term on the right-hand side) and then utilizing \eqref{smart1}-\eqref{smart3} we obtain
\[\begin{split}
 2^{-j (5- 4\alpha +\varepsilon )} &= \frac{4}{3} \left( u_{Q}(t_0)^2 - u_{Q}(t_1)^2 \right) \\
&\leq  c \int_{t_1}^{t_0} u_{Q} \left( 2^j u_{3Q/2,j\pm 2} \sum_{\theta j \leq k\leq j-5 } 2^{3k/2} u_{3Q/2,k}  + 2^{5j/2}  u_{3Q/2,j\pm 4}^2 \right. \\
&\hspace{2cm}\left.+ 2^{3j/2} \sum_{k\geq j+1} 2^{k} u_{3Q/2,k}^2  \right) + c\, 2^{j(2\alpha -\varepsilon )} \int_{t_1}^{t_0} u_{3Q/2,j\pm 2}^2 + e(j) \\
&\leq c\, 2^{-\frac{j}{2}(5-4\alpha +\varepsilon ) } \left( 2^j 2^{-\frac{j}{2} (5-2\alpha + \varepsilon )} 2^{-\frac{j}{2} (2-2\alpha + \varepsilon/2 )} + 2^{5j/2} 2^{-j (5-2\alpha +\varepsilon )} \right.\\
&\hspace{2cm}\left.+2^{3j/2}2^{-j(4-2\alpha + \varepsilon )} \right) + c\, 2^{j(2\alpha -\varepsilon )} 2^{-j(5-2\alpha +\varepsilon )}\\
&\leq c\, 2^{-j (5-4\alpha +\varepsilon )} \left( 2^{-3j\varepsilon/8 } + 2^{-j\varepsilon /2 } +2^{-j\varepsilon /2 } + 2^{-3j\varepsilon/2 } \right) \\
&\leq c\, 2^{-j (5-4\alpha +\varepsilon )}  2^{-3j\varepsilon /8 }
\end{split}
\]
where, in the second inequality, we also used the Cauchy-Schwarz inequality and used the inequality $j \leq c \,2^{j\varepsilon /4}$, as well as absorbed $e(j)$ (by writing, for example, $e(j)\leq c \,2^{-j(5-4\alpha +2\varepsilon )}$ (recall the beginning of Section \ref{sec_prelims} for the definition of the $j$-negligible error $e(j)$)). Thus
\[
1 \leq c\, 2^{-j\varepsilon /4 },
\]
which gives a contradiction for sufficiently large $j$.
\end{proof}
\subsection{The singular set}\label{sec_sing_set}

Having defined good cubes and bad cubes, and observing that we have a ``slightly more than critical'' estimate on a cube that has some good ancestors  (Theorem \ref{thm_j_good_gives_critical regularity}), we now characterize the singular set $S$ in terms of its covers by bad cubes, and (in the next section) we show a much stronger (than critical) estimate regularity outside $S$.

Let $A_j$ denote the union of all bad $j$-cubes. Using Vitali Covering Lemma we can find a cover $\mathcal{A}_j$ that covers $A_j$ and such that
\eqnb\label{card_of_Aj}
\# \mathcal{A}_j \leq c \, 2^{j(5-4\alpha + \varepsilon )}.
\eqne
Indeed, the Vitali Covering Lemma gives a sequence of pairwise disjoint bad $j$-cubes $Q^{(l)}$ such that 
\[
A_j \subset \bigcup_l 5Q^{(l)}
\]
However, since $\int_0^\infty \int |(-\Delta )^{\alpha /2} u |^2 \leq c$ (from the energy inequality, recall \eqref{EI_prelims}), 
\eqnb\label{calc_no_of_Aj}\begin{split}
c&\geq \int_0^\infty \int | \xi |^{2\alpha} \left| \widehat{u} (\xi ) \right|^2\\
&= \sum_{k\in \ZZ}  \int_0^\infty \int   p_k (\xi ) | \xi |^{2\alpha} \left| \widehat{u} (\xi ) \right|^2 \\
&\geq c\sum_{k\geq j} 2^{2\alpha k} \int_0^\infty \int  p_k (\xi )^2  \left| \widehat{u} (\xi ) \right|^2 \\
&=c \int_0^\infty \int \sum_{k\geq j} 2^{2\alpha k} | P_k u |^2  \\
&\geq c\sum_l \int_0^\infty \int_{Q^{(l)}} \sum_{k\geq j} 2^{2\alpha k} | P_k u |^2 \\
&\geq c\sum_l 2^{-j(5-4\alpha +\varepsilon )} ,
\end{split}
\eqne
where we used the Plancherel identity (twice, in the first and fourth lines), Tonelli's theorem (twice, in the second and fourth lines), the fact that $Q^{(l)}$'s are pairwise disjoint in the fifth line. Thus
\[
l\leq c2^{j(5-4\alpha +\varepsilon )},
\]
and so $\mathcal{A}_j$ can be obtained by covering each of $5Q^{(l)}$ by at most $6^3$ $j$-cubes.

In the remainder of this section we will show that there exists a (larger) $j$-cover ${\mathcal{B}}_j$ of all bad $j$-cubes (i.e. of $A_j$) with the same cardinality (i.e. satisfying \eqref{card_of_Aj}, but with a larger constant) and the additional property that
\eqnb\label{barrier_property}
\begin{split}
\text{for any }x\text{ outside of }{\mathcal{B}}_j& \text{ there exists }r\in (0,2^{-10}) \text{ such that }\p (rQ_j (x))\text{ does not}\\
&\text{ touch any bad }k\text{-cube for any }k\geq j. 
\end{split}
\eqne 
(Recall that $Q_j(x)$ denotes the $j$-cube centered at $x$.) We will refer to $\p (rQ_j (x)) $ as the \emph{barrier}, and to \eqref{barrier_property} as the \emph{barrier property}.
We first discuss a simple geometric lemma.
\begin{lemma}[Geometric Lemma]\label{lem_geom}
Let $Q=Q(y)$, $Q'=Q'(x)$ be open cubes with sidelengths $2a$, $2b$, respectively. Then 
\[\p (rQ )\text{ intersects }Q' \Rightarrow r\in [r_{Q'} - b/a,r_{Q'}+b/a],\]
where $r_{Q'} >0$ is such that $x\in \p (r_{Q'}Q)$.
\end{lemma}
\begin{proof}
\begin{figure}[h]
\centering
 \includegraphics[width=\textwidth]{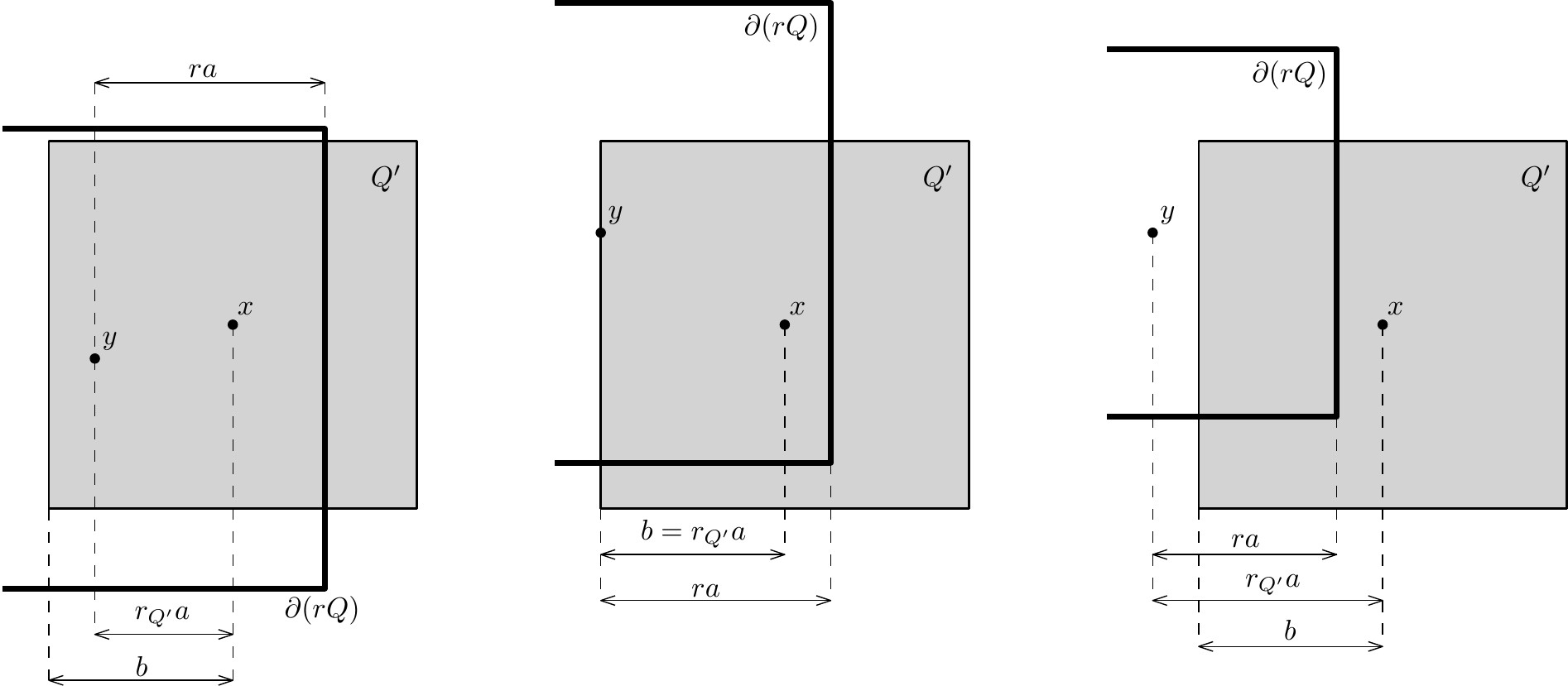}
 \nopagebreak
  \captionof{figure}{Sketch of the interpretation of Lemma \ref{lem_geom}.}\label{fig_lem_geom} 
\end{figure}
We will write $\gamma \coloneqq b/a$ for brevity. We split the reasoning into cases.\vspace{0.3cm}\\
\emph{Case 1.} $y\in \p Q'$.

Then $r_{Q'}= b/a$ (see Figure \ref{fig_lem_geom} (middle)) and so $r\geq r_{Q'}-b/a$ trivially. Moreover $\p (rQ )$ intersects $Q'$ if and only if $ra<2b$ (see Figure \ref{fig_lem_geom} (middle)), that is $r<2b/a=r_{Q'}+b/a$, as required.\vspace{0.3cm}\\
\emph{Case 2.}  $y\not \in \overline{Q'}$. 

Then $r_{Q'}>b/a$ (which is clear by comparison with Case 1), and $\p (rQ )$ intersects $Q'$ if and only if
\[
r_{Q'}a-b < ra < r_{Q'}a +b
\]
(see Figure \ref{fig_lem_geom} (right)), as required.\vspace{0.3cm}\\
\emph{Case 3.} $y\in Q'$.

Then $r_{Q'} < b/a$ and $\p (rQ)$ intersects $Q'$ if and only if 
\[
b-r_{Q'} a < ra < r_{Q'} a + b
\]
(see Figure \ref{fig_lem_geom} (left)). The claim follows by ignoring the first of these two inequalities (and writing $r\geq 0 > r_{Q'}-b/a$ instead).
\end{proof}
We can now construct the $j$-cover satisfying the barrier property \eqref{barrier_property}.
\begin{lemma}
For every $j\geq 0$ there exists a $j$-cover ${\mathcal{B}}_j$ of $A_j$ such that $\# {\mathcal{B}}_j \leq c \, 2^{j(5-4\alpha + \varepsilon )}$ and the barrier property \eqref{barrier_property} holds.
\end{lemma}
\begin{proof} (Here we follow the argument from \cite{katz_pavlovic}.)
We will find a $j$-cover (also denoted by ${\mathcal{B}}_j$) of $A_j$ such that 
\eqnb\label{barrier_prop_alt}
\begin{split}
\text{ for any }j\text{-cube }Q\text{ outside of }{\mathcal{B}}_j&\text{ there exists }r\in (0,2^{-10})\text{ such that }\p (rQ)\text{ does not}\\
&\text{ touch any bad }k\text{-cube for any }k\geq j.
\end{split}
\eqne
(Here ``outside'' is a short-hand notation for ``disjoint with every element of''.)
The barrier property \eqref{barrier_property} is then recovered by replacing every $j$-cube $Q\in {\mathcal{B}}_j$ by $3Q$ and covering it by at most $4^3$ $j$-cubes. Indeed, then for any $x$ outside of such set we have that $Q_j(x)$ (the $j$-cube centered at $x$) is outside of ${\mathcal{B}}_j$ and so the barrier property \eqref{barrier_property} follows from \eqref{barrier_prop_alt}.\vspace{0.4cm}\\
\emph{Step~1.} We define \emph{naughty} $j$-cubes.\\

We say that a $j$-cube $Q$ is $k$-\emph{naughty}, for $k\geq j$, if it intersects more than $\eta 2^{(k-j)(5-4\alpha +2\varepsilon )}$ elements of  ${\mathcal{A}_k}$. Here $\eta \in (0,1)$ is a universal constant, whose value we fix in Step~4 below. We say that a $j$-cube is \emph{naughty} if it is $k$-naughty for any $k\geq j$. (Note that a bad cube is naughty. A good cube is not necessarily naughty, and vice versa.)\vspace{0.4cm}\\
\emph{Step~2.} For each $k\geq j$ we construct a $j$-cover $\mathcal{B}_{j,k}$ of all $k$-naughty $j$-cubes, such that
\eqnb\label{bound_on_cardinality_of_B_jk}
\#  \mathcal{B}_{j,k}  \leq c \eta^{-1} 2^{j(5-4\alpha + \varepsilon)} 2^{\varepsilon (j-k)}.
\eqne
(Note that $\mathcal{B}_{j,j}$ covers all $j$-naughty $j$-cubes, and so in particular all bad $j$-cubes.)\\

Let $Q^{(1)}$ be any $k$-naughty $j$-cube. Given $Q^{(1)}, \ldots , Q^{(l)}$ let $Q^{(l+1)}$ be any $k$-naughty $j$-cube that is disjoint with with each of $3Q^{(1)}, \ldots , 3Q^{(l)}$. Note that then $3Q^{(1)}, \ldots , 3Q^{(l)}$ contain all elements of $\mathcal{A}_k$ that $Q^{(1)}, \ldots , Q^{(l)}$ intersect. This means that $Q^{(l+1)}$ intersects at least $\eta \,2^{(k-j)(5-4\alpha +2\varepsilon )}$ ``new'' elements of $\mathcal{A}_k$ (i.e. the elements that none of $Q^{(1)}, \ldots ,Q^{(l)}$ intersect). This means that such an iterative definition can go on for at most
\[
L\coloneqq \# {\mathcal{A}_k}/ \eta 2^{(k-j)(5-4\alpha +2\varepsilon )} \leq c \eta^{-1} 2^{j(5-4\alpha + \varepsilon)} 2^{\varepsilon (j-k)}
\]
steps, and then the family $\{ 3Q^{(1)}, \ldots , 3Q^{(L)}\}$ covers all $k$-naughty $j$-cubes. We now cover each of $3Q^{(l)}$ ($l=1,\ldots , L$) by at most $4^3$ $j$-cubes to obtain $\mathcal{B}_{j,k}$. (Note \eqref{bound_on_cardinality_of_B_jk} then follows from the upper bound on $L$.)\vspace{0.4cm}\\
\emph{Step~3.} We define $\mathcal{B}_j$.\\

Let
\[
\mathcal{B}_j \coloneqq \bigcup_{k\geq j} \mathcal{B}_{j,k}.
\]
By construction, $ \mathcal{B}_j$ covers all naughty $j$-cubes (and so, in particular, all bad $j$-cubes) and 
\[\#  \mathcal{B}_{j} \leq \sum_{k\geq j}\#  \mathcal{B}_{j,k}  \leq  c \eta^{-1} 2^{j(5-4\alpha + \varepsilon)} \sum_{k\geq j}2^{\varepsilon (j-k)} =c \eta^{-1} 2^{j(5-4\alpha + \varepsilon)},
\]
as required (given $\eta $ is fixed).\vspace{0.4cm}\\
\emph{Step~4.} We show that \eqref{barrier_prop_alt} holds for sufficiently small $\eta \in (0,1)$. (This, together with the previous step, finishes the proof.)\\

Let $Q$ be a $j$-cube disjoint with all elements of $\mathcal{B}_j$.  Let us denote by $\mathcal{C}^k (Q)$ the collection of $k$-cubes $Q'$ ($k\geq j$) from $\mathcal{A}_k$ intersecting $Q$. Since $Q$ is not naughty (as otherwise it would be covered by $\mathcal{B}_j$) 
\[  \# \mathcal{C}^k (Q)\leq \eta 2^{(k-j)(5-4 \alpha + 2 \varepsilon )}.\] 

Let $r_{Q'}\in (0,\infty )$ be such that $\p (r_{Q'} Q)$ contains the center of $Q'$. Applying Lemma \ref{lem_geom} with $2a= 2^{-j(1-\varepsilon )}$ and $2b=2^{-k(1-\varepsilon )}$ we obtain that 
\[\p (r Q) \text{ intersects  }Q' \Rightarrow r \in [r_{Q'} - 2^{(1-\varepsilon )(j-k)},r_{Q'} + 2^{(1-\varepsilon )(j-k)} ].
\] Thus if $f_k (r)$ denotes the number of bad $k$-cubes that intersect $\p (rQ )$ then
\[
f_k (r) \leq \sum_{Q'\in \mathcal{C}^k (Q)} \chi_{\left[ r_{Q'} - 2^{(1-\varepsilon )(j-k)},r_{Q'} + 2^{(1-\varepsilon )(j-k)} \right]} (r).
\]
Thus
\[
\| f_k \|_{L^1 (0 ,2^{-10})} \leq 2 \#  \mathcal{C}^k (Q) 2^{(1-\varepsilon )(j-k)} \leq 2 \eta 2^{(4\alpha -4 -3\varepsilon )(j-k)},
\]
and so letting $f\coloneqq \sum_{k\geq j} f_k$ and recalling that $\alpha >1$ and $\varepsilon $ is small enough so that $4\alpha - 4-3\varepsilon >0$ (see \eqref{how_small_is_eps}) we obtain
\[
\| f \|_{L^1 (0 ,2^{-10})}\leq \sum_{k\geq j}\| f_k \|_{L^1 (0 ,2^{-10})} \leq c\eta .
\]
(This is the only place in the article where we need the assumption $\alpha >1$; otherwise $\alpha \geq 1$ would be sufficient.)
By choosing $\eta \in (0,1)$ sufficiently small such that $c\eta < 2^{-10}/2$ we see that $\| f \|_{L^1 (0 ,2^{-10})}< 2^{-10}$, and so there exists $r\in (0 ,2^{-10})$ such that $f(r)=0$ (recall that $f$ takes only integer values). In other words there exists $r$ such that $\p (rQ)$ does not intersect any element of $\mathcal{A}_k$ for any $k\geq j$, and so in particular any bad $k$-cube.
  \end{proof}
We now let 
\[
E \coloneqq \limsup_{j\to \infty } \bigcup_{Q\in \mathcal{B}_j} Q.
\]
Observe that, since $\# \mathcal{B}_j \leq c\, 2^{j(5-4\alpha + \varepsilon )}$, 
\[
d_H (E) \leq 5-4\alpha + \varepsilon ,
\]
see, for example, Lemma 3.1 in \cite{katz_pavlovic} for a proof.

\subsection{Regularity outside $E$}\label{sec_reg_outside_E}
We now show that for every $x\not \in E$ and every interval of regularity $(a_i,b_i)$ there exists an open neighbourhood of $x$ on which $u(t)$ remains bounded (as $t\in ((a_i+b_i)/2,b_i) $). This together with the above bound on $d_H(S)$ finishes the proof of Theorem \ref{thm_main}.

Note that if $x\not \in E$ then for sufficiently large $j_0$
\[
x\not \in Q\quad \text{ for any } Q\in \mathcal{B}_j \text{ for }j\geq j_0.
\]
In particular
\eqnb\label{barrier_conseq1}
x\text{ does not belong to any bad }j\text{-cube for }j\geq j_0
\eqne
(since $\mathcal{B}_j$ is a cover of all bad $j$-cubes), and
for any $j_1\geq j_0$ there exists $r=r(x,j_1)\in (0 ,2^{-10})$ such that
\eqnb\label{barrier_conseq2}
\p (rQ_{j_1}(x)) \quad \text{ does not intersect any bad }k\text{-cube with }k\geq j_1
\eqne
(by the barrier property, \eqref{barrier_property}).
The point is that the barrier can be constructed for any $j_1\geq j_0$. This will be relevant for us, since in the proof of regularity below we will consider a $j$-cube with $j\geq j_1\geq j_0/\theta^2 $. Thus we will be able to deal with some of the low modes ($k\in [ \theta  j, j-5 ] )$) using \eqref{barrier_conseq1} and other using \eqref{barrier_conseq2}. Indeed, for such modes we will have ``cubes larger than $j$-cube'' (i.e. $Q_k$ with $k<j$) and we will obtain the critical decay on such cubes by either utilising the barrier property \eqref{barrier_conseq2} (for cubes that are only ``a little bit larger'', see Case 1 in Step~2 for details) or the fact that distant ancestors are large enough to contain $x$ so that we can use \eqref{barrier_conseq1}. As for local and high modes (i.e. $k\geq j-5$) we will use the barrier property \eqref{barrier_conseq2} to obtain critical regularity for cubes located near the barrier, with more and more regularity on cubes located further away from the barrier towards the interior. In fact we can guarantee arbitrary strong estimate for cubes located sufficiently far from the barrier, but we limit ourselves to the estimate $\lesssim 2^{-j(5-4\alpha + 10)/2}$.

We now proceed to a rigorous version of the above explanation.
\begin{theorem}[Regularity outside $E$]\label{thm_reg_outside_E}
Let $x\not \in E$. Given an interval of regularity $(a_i,b_i)$ there exists $c_i>1$ and  $j_1=j_1(c_i) \in \NN $ such that
\eqnb\label{better_than_critical_decay}
u_Q (t) < c_i 2^{-j \rho (Q) /2}
\eqne
for all $t\in ((a_i+b_i)/2,b_i)$ and for every $j$-cube $Q\subset r Q_{j_1} (x)$, where $r\in (0,2^{-10}) $ is as in \eqref{barrier_conseq2}, 
 \[
 \rho (Q) \coloneqq 5-4\alpha + \min (10, \varepsilon \delta (Q) /10 )
 \]
 and $\delta (Q)$ denotes the smallest $k\in \NN$ such that $Q_{j-k}$ intersects $\p (rQ_{j_1} (x))$.
\end{theorem}
Note that the theorem gives no restriction on the range of $j$'s, but it is clear from the inclusion $Q\subset r Q_{j_1} (x)$ that $j\geq j_1 + 10$ (as $r<2^{-10}$).
\begin{proof}
Since $u$ is a strong solution in $(a_i,b_i)$, it is continuous in time in $(a_i,b_i)$ with values in $H^6$ (recall \eqref{int_of_reg}). Thus letting \[
c_i \coloneqq 1+ c\left\| u\left(\frac{a_i+b_i}2 \right) \right\|_{H^6}
\] we see that, for any $j$-cube $Q$, $u_Q ((a_i+b_i)/2) \leq \| P_j u ((a_i+b_i)/2) \| < c_i 2^{-6j} $, and hence also $u_Q(t) <c_i$ for some $t>(a_i+b_i)/2$ (due to the continuity of the $H^6$ norm). Thus the claim remains valid on some nonempty time interval following $(a_i+b_i)/2$ (since $\rho (Q) \leq 5-4\alpha +10\leq 11$). 

Since the interval of regularity $(a_i,b_i)$ is fixed, from now on we will suppress the subindex ``$i$'', for brevity. 

We take $j_0$ sufficiently large so that \eqref{barrier_conseq1} and the claims of Corollary~\ref{cor_basic_est_for_jcubes} and Theorem~\ref{thm_j_good_gives_critical regularity} are valid (we will let $j_0$ even larger below). We let $j_1$ be the smallest integer such that
\eqnb\label{choice_of_j1}
j_1 \geq (j_0 +10)/\theta^2.
\eqne

We also note that if 
\eqnb\label{q'_fact}
\begin{split}
Q'(y)\text{ is a }k\text{-cube centered at }y\in rQ_{j_1} (x) &\text{ and touching the barrier }\p (rQ_{j_1}(x))\\
&\text{ then }Q'\text{ is good if }k\geq j_0. 
\end{split}
\eqne
Indeed, if $k\geq j_1$ then $Q'$ is good by the barrier property \eqref{barrier_conseq2}. If $k<j_1 $ then $Q' \supset r Q_{j_1} (x)\ni x$ (as the sidelength of $Q'(y)$ is more than $2^{10}$ times larger than the sidelength of $r Q_{j_1} (x)\ni y$), and so $Q'$ is good by \eqref{barrier_conseq1}.\\

Suppose that the theorem is false and let $t_0> (a+b)/2$ be the first time when it fails. Then 
\eqnb\label{ok_until_t0}
u_{Q'} (t) \leq  c \,2^{-k \rho (Q') /2} \quad \text{ for all }t\in [0,t_0]\text{ and all }k\text{-cubes }Q'\subset r Q_{j_1} (x)
\eqne
and there exists a $j$-cube $Q\subset r Q_{j_1} (x)$ (for some $j\geq 0$) such that 
\eqnb\label{fail_at_t0}
u_Q (t_0) \geq  2^{-j\rho (Q)/2}.
\eqne
We note that the existence of such $Q$ is nontrivial, since there are infinitely many functions $u_{Q'}(t)$ for $Q'\subset r Q_{j_1} (x)$. In fact one can think of a scenario when all such $u_{Q'}$'s remain close to zero until $t_0$ with a sequence of $u_{Q'}$'s growing faster and faster past $t_0$ (in such scenario \eqref{ok_until_t0} holds but not \eqref{fail_at_t0}). We verify in Step~1 below that such a scenario does not happen (i.e. that such $Q$ exists) as long as $t_0$ lies inside $(a,b)$.\footnote{This is the localisation issue that we have referred to in the introduction. This issue has been ignored in \cite{katz_pavlovic}.}

We now let $t_1 \in (0,t_0)$ be the last time such that $u_Q(t_1) = \frac{1}{2} 2^{-j \rho (Q)/2 }$. Then
\eqnb\label{uQ_between_t1_t0}
u_Q (t) \in  [2^{-j\rho(Q)/2}/2, 2^{-j\rho(Q)/2} ] \quad \text{ for } t\in [t_1,t_0].
\eqne
The main estimate \eqref{basic_estimate_jcubes} gives
\eqnb\label{basic_est_rewritten_for_full_reg}\begin{split}
2^{-j\rho (Q) } &= \frac{4}{3} \left( u_Q(t_0)^2 - u_Q(t_1)^2 \right) \\
&\leq  -c\,2^{2\alpha j }\int_{t_1}^{t_0}  u_{Q}^2+c \int_{t_1}^{t_0} u_{Q} \left( 2^j u_{3Q/2 ,j\pm 2} \sum_{\theta  j \leq k\leq j-5 } 2^{3k/2} u_{Q_k}   \right. \\
&\hspace{2cm}\left.+ 2^{5j/2}  u_{3Q/2 ,j\pm 4}^2+ 2^{3j/2} \sum_{k\geq j+1} 2^{k} u_{3Q/2 ,k}^2  \right) \\
&\hspace{2cm}+ c \int_{t_1}^{t_0} 2^{2\alpha j } 2^{-j\varepsilon } u_{3Q/2,j\pm 2}^2   +e(j) ,
\end{split}
\eqne
where we omitted time argument in our notation. Note that we can write 
\[
e(j) \leq c\, 2^{-20j}
\]  
(recall the beginning of Section \ref{sec_prelims} for the definition of $e(j)$, the $j$-negligible error), so that it can be ignored (i.e. it can be absorbed into the left-hand side for sufficiently large $j$). We will estimate the terms appearing on the right-hand side of \eqref{basic_est_rewritten_for_full_reg} in steps 2-4 below, and we will conclude the proof in Step~5. \vspace{0.4cm}\\
\emph{Step~1.} We verify \eqref{fail_at_t0}. 

Let $m\in \NN$. By definition of $t_0$ there exists $\tau \in (t_0,t_0+1/m)$ and a $j$-cube ${Q}$ such that $u_{Q} (\tau )\geq c\,2^{-j\rho ({Q})/2}$. We claim that \eqref{fail_at_t0} holds for such ${Q}$ if $m$ is taken sufficiently large. Indeed, if it does not, then $2^{j\rho ({Q})/2} u_{Q}(t_0 )\leq 1$ for each $m$, and so
\[
c-1 \leq 2^{j\rho ({Q})/2}  \left( u_{Q} (\tau ) - u_{Q} (t_0) \right) \leq  2^{11j/2} \left\|  \phi_{Q}^2  P_j ( u(\tau ) - u (t_0 ) )  \right\|  \leq c \| u(\tau)-u(t_0) \|^2_{H^6 (\RR^3)}
\]
for all $m$, uniformly in $j$, and so continuity of $u$ in time (on $(a,b)$) with values in $H^6$ gives a contradiction for sufficiently large $m$. (Note that, for simplicity, we have omitted the dependence of $\tau$ and $Q$ on $m$ in the notation above.)  \vspace{0.4cm}\\
\emph{Step~2.} We observe that $\delta (Q) \geq 11$, so that in particular 
\eqnb\label{rho(Q)_geq_5-4alpha+eps}
\rho (Q) \geq 5-4\alpha + \varepsilon.
\eqne

In order to see this note that if $\delta (Q) \leq 10$ then $Q_{j-10}$ touches $\partial ( r Q_{j_1} (x))$. Thus \eqref{q'_fact} implies that $Q_{k-10}$ is good for every $k\in [\theta j,j ]$, since 
\[k-10\geq \theta j -10 \geq \theta j_1 - 10 \geq j_0 \]
by our choice \eqref{choice_of_j1} of $j_1$. Hence Theorem \ref{thm_j_good_gives_critical regularity} gives that
\[
2u_Q (t_0)< 2^{-j(5-4\alpha + \varepsilon )/2} \leq 2^{-j(5-4\alpha + \varepsilon \delta (Q)/10 )/2}  = 2^{-j \rho (Q)/2},
\]
which contradicts \eqref{fail_at_t0}.\vspace{0.4cm}\\
\emph{Step~3.} We show that
\eqnb\label{est_on_uk}
\begin{split}
u_{Q_k} (t) &\leq c \,2^{-k(5-4\alpha + \varepsilon )/2}\qquad k\in [\theta j, j-5 ],\\
u_{3Q/2,k} (t) &\leq \begin{cases}
c \,2^{-j(\rho(Q) -2\varepsilon/5 )/2} &k\in [ j-4 ,\ldots , j+ 100/\varepsilon ],	\\
c\,2^{-3j}2^{-k(9-4\alpha)/2}& k\geq j+ 100/\varepsilon
\end{cases}
\end{split}
\eqne
for $t\in (t_1,t_0)$.\vspace{0.3cm}\\
\emph{Case 1.} $k\in [\theta j, j-5 ]$. 

If $\delta (Q_k ) \geq 11$ then in particular $Q_k \subset rQ_{j_1}(x) $ and $\rho (Q_k) \geq 5-4\alpha +\varepsilon $ and so the claim follows from \eqref{ok_until_t0}. If $\delta (Q_k) \leq 10$ then $Q_{l-10}$ is good for every $l\in [\theta k, k]$ due to \eqref{q'_fact}, since
\eqnb\label{temp_lk}
l-10\geq \theta k -10 \geq \theta^2 j - 10 \geq \theta^2 j_1 -10 \geq j_0.
\eqne
Therefore the claim follows from Theorem \ref{thm_j_good_gives_critical regularity}.\vspace{0.3cm}\\
\emph{Case 2.} $k\in [j-4, \ldots , j+100/\varepsilon )$. 

Then
\eqnb\label{k_geq_j-4_deltas}
\delta (Q_k) = \delta (Q) +k-j \geq \delta (Q) -4\geq 7,
\eqne 
where we used Step~2 in the last inequality. Hence $Q_k\subset rQ_{j_1}(x) $ and
\[
\rho (Q_k ) \geq \rho (Q) -2\varepsilon /5.
\] 
Thus since for $k\in [j-4, j-1 ]$ we have $3Q/2\subset Q_k$, \eqref{ok_until_t0} gives
\[
u_{3Q/2,k} \leq 2^{-k \rho (Q_k ) /2} \leq 2^{-k(\rho(Q) -2\varepsilon/5 )/2} \leq c\, 2^{-j(\rho(Q) -2\varepsilon/5 )/2} ,
\]
as required. If $k\geq j$ we note that 
\eqnb\label{cover_of_7Q/4}
u_{3Q/2,k} \leq \sum_{Q'\in S_k (7Q/4 )} u_{Q'} ,
\eqne
where $S_k (7Q/4)$ denotes a cover of $7Q/4$ by $k$-cubes with $\# S_k (7Q/4)\leq c 2^{3(k-j)(1-\varepsilon )}$ (recall the beginning of Section \ref{sec_main_result}). Since 
\eqnb\label{q'_vs_7Q4}
Q'_j=2^{-(j-k)(1-\varepsilon )}Q' \subset Q_{j-2}\quad \text{ for every }Q' \in S_k (7Q/4 ),
\eqne
see Fig.~\ref{fig_cover_of_7Q4}, we obtain
\eqnb\label{rel_between_Q'_and_Q_for_k_leq_100j/eps}
\delta (Q' ) = \delta (Q'_j ) + k-j \geq \delta (Q_{j-2}) = \delta (Q) -2,
\eqne
and so $\rho (Q') \geq \rho (Q) -\varepsilon/5 $.
\begin{figure}[h]
\centering
\captionsetup{width=.9\linewidth}
 \includegraphics[width=0.3\textwidth]{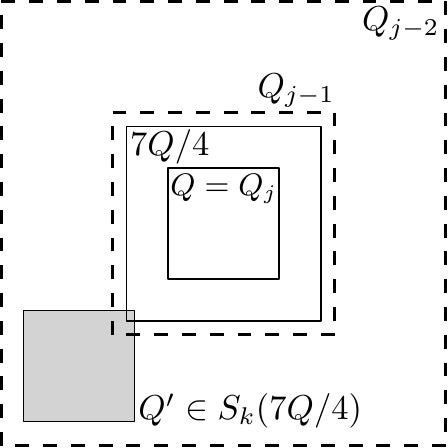}
 \nopagebreak
  \captionof{figure}{An illustration of \eqref{q'_vs_7Q4} - note that each $Q'\in S_k (7Q/4)$ is of the same size as $Q$ (as in the illustration) or smaller (as $k\geq j$).} \label{fig_cover_of_7Q4} 
\end{figure} Therefore \eqref{ok_until_t0} gives
\[
u_{Q'}\leq 2^{-k\rho (Q' )/2} \leq 2^{-k ( \rho (Q) -\varepsilon /5)/2 } \leq c \, 2^{-j ( \rho (Q) -2\varepsilon /5)/2 } ,
\]
and since $\# S_k (7Q/4 ) \leq c\,2^{300(1-\varepsilon)/\varepsilon } = c$ (recall our constants may depend on $\varepsilon $) the claim follows by applying \eqref{cover_of_7Q/4} above.\vspace{0.3cm}\\
\emph{Case 3.} $k \geq j+100 /\varepsilon$. 

For such $k$ we improve \eqref{rel_between_Q'_and_Q_for_k_leq_100j/eps} by writing
\eqnb\label{rel_between_Q'_and_Q_for_k_geq_100j/eps}
\delta (Q' ) = \delta (Q'_j ) + k-j \geq \delta (Q_{j-2}) + 100/\varepsilon  = \delta (Q) +100/\varepsilon -2> 100/\varepsilon
\eqne
for any $Q'\in S_k (7Q/4)$
where we used Step~2 in the last inequality. This gives $\rho (Q') = 15-4\alpha $. Thus using \eqref{cover_of_7Q/4} and the estimate $\# S_k (7Q/4 )\leq c\,2^{3(k-j )(1-\varepsilon )} \leq c\,2^{3(k-j )} $ we arrive at
\[
u_{3Q/2,k} \leq \sum_{Q' \in S_k (7Q/4)} u_{Q'} \leq  \sum_{Q' \in S_k (7Q/4)} 2^{-k\rho(Q')/2} \leq c\, 2^{3(k-j)} 2^{-k\rho (Q')/2} = c \, 2^{-3j} 2^{-k(9-4\alpha )/2},
\]
as required.\vspace{0.4cm}\\
\emph{Step~4.} We use the previous step to estimate the terms appearing on the right-hand side of the main estimate \eqref{basic_est_rewritten_for_full_reg}. Namely we show that
\eqnb\label{est_terms_on_rhs}\begin{split}
\sum_{\theta j \leq k \leq j-5}2^{3k/2} u_{3Q/2,k} &\leq c\, 2^{3j/2 }2^{-j (5-4\alpha )/2}2^{-j \varepsilon /2},\\
  u_{3Q/2,j\pm 4}^2 &\leq c\, 2^{-j\rho (Q) /2 }2^{-j (5-4\alpha )/2} 2^{-j\varepsilon /10 } ,\\
    \sum_{k\geq j+1 } 2^k u_{3Q/2,k}^2 &\leq c \, 2^{j}2^{-j\rho (Q)/2}2^{-j(5-4\alpha )/2} 2^{-j  \varepsilon/10}
    \end{split}
\eqne

We note that, although the terms appearing on the right-hand side might look complicated we write them in this form to articulate their roles. As for the factors $2^{3j/2}$ or $2^j$, these are ``bad factors'' which, together with the corresponding factor in the main estimate \eqref{basic_est_rewritten_for_full_reg}, give $2^{5j/2}$. This should be compared against the factor $2^{2\alpha j}$ which is a ``good factor'' given by the dissipation (i.e. by the first term on the right-hand side of \eqref{basic_est_rewritten_for_full_reg}, which comes with a minus). This brings us to the factors of the form $2^{-j(5-4\alpha  )}$ whose role is exactly to balance the ``bad factor'' against the ``good factors''. 

As for the factors $2^{-j\rho (Q)/2}$, we point out that together with the corresponding factor $u_{Q}$ (which is bounded above and below by $2^{-j\rho(Q)/2}$ due to \eqref{uQ_between_t1_t0}) appearing in the basic estimate, one obtains $2^{-j\rho (Q)}$ as the common factor of all terms in \eqref{basic_est_rewritten_for_full_reg}.

Finally, the role of any factor involving $\varepsilon$ is to make sure that the balance falls in our favor, namely that the resulting constant at all terms on the right-hand side of \eqref{basic_est_rewritten_for_full_reg} (except for the first term), is smaller than the constant at the first term (the dissipation term). Writing the estimates in this way also points out the appearance of $5-4\alpha $, which is our desired bound on the Hausdorff dimension.\\

We now briefly verify \eqref{est_terms_on_rhs}. The first two of them follow from Step~3 by a simple calculation,
\eqnb\label{where_to_plugin_low_modes}
\sum_{\theta j \leq k \leq j-5}2^{3k/2} u_{3Q/2,k} \leq c\sum_{\theta  j \leq k \leq j-5}2^{-k (2-4\alpha +\varepsilon )/2} \leq c \, 2^{-j (2-4\alpha +\varepsilon )/2}
\eqne
and
\[
 u_{3Q/2,j\pm 4}^2 \leq c\, 2^{-j (\rho (Q) -2\varepsilon /5)}= c\, 2^{-j\rho(Q)/2} 2^{-j(\rho (Q) - 4\varepsilon /5 )/2} \leq c \, 2^{-j\rho(Q)/2} 2^{-j (5-4\alpha )/2} 2^{-j \varepsilon /10}, 
\]
as required, where we used \eqref{rho(Q)_geq_5-4alpha+eps} in the last inequality. As for the third estimate in \eqref{est_terms_on_rhs} we write $\sum_{k\geq j+1} = \sum_{j+1\leq k \leq j+100/\varepsilon } + \sum_{k>j+100/\varepsilon}$, and estimate each of the two sums separately,
\[
\sum_{j+1\leq k \leq j+100/\varepsilon } 2^k u_{3Q/2,k}^2 \leq c \, 2^j 2^{-j (\rho (Q) -2\varepsilon /5)} \leq c\,2^j 2^{-j \rho(Q)/2} 2^{-j (5-4\alpha + \varepsilon/5)/2}
\]
(recall that $c$ might depend of $\varepsilon $), where we used \eqref{rho(Q)_geq_5-4alpha+eps} in the last inequality, and
\[
\sum_{k>j+100/\varepsilon} 2^k u_{3Q/2,k}^2 \leq c\, 2^{-3j} \sum_{k>j+100/\varepsilon} 2^{-k(8-4\alpha )} \leq c\, 2^{-j(11-4\alpha )}\leq c\,2^j 2^{-j \rho(Q)/2} 2^{-j (5-4\alpha + \varepsilon/5)/2},
\]
where we used the inequality $11-4\alpha \geq -1 +\rho(Q)/2 +(5-4\alpha )/2+\varepsilon/10$ (a trivial consequence of the fact that $\rho (Q) \leq 5-4\alpha +10$).\vspace{0.4cm}\\
\emph{Step~5.} We conclude the proof.

Applying the estimates from the previous step into the main estimate \eqref{basic_est_rewritten_for_full_reg} and recalling that $u_{3Q/2,j\pm 2}^2 \leq c\, 2^{-j(\rho(Q) -2\varepsilon /5 )}$ (from Step~3) we obtain
\[\begin{split}
2^{-j\rho (Q) } &\leq  -c\,2^{2\alpha j }\int_{t_1}^{t_0}  u_{Q}^2\\
&+c \int_{t_1}^{t_0} u_{Q} \left( 2^j 2^{-j(\rho(Q) -2\varepsilon /5 )/2} 2^{3j/2 }2^{-j (5-4\alpha )/2}2^{-j \varepsilon /2} + 2^{5j/2}  2^{-j\rho (Q) /2 }2^{-j (5-4\alpha )/2}2^{-j\varepsilon /10 }  \right. \\
&\hspace{5cm}\left.+ 2^{3j/2} 2^{j}2^{-j\rho (Q)/2}2^{-j(5-4\alpha )/2} 2^{-j  \varepsilon/10} \right) \\
&+ 2^{2\alpha j}2^{-j\varepsilon }\int_{t_1}^{t_0} 2^{-j(\rho(Q) -2\varepsilon /5 )}\\
&= -c\,2^{2\alpha j }\int_{t_1}^{t_0}  u_{Q}^2 + c \, 2^{2\alpha j } \int_{t_1}^{t_0} u_Q \left( 2^{-j\rho(Q)/2}(2^{-3j\varepsilon /10} + 2^{-j\varepsilon /10 } + 2^{-j\varepsilon /10 } )\right)\\
&+c\,  2^{2\alpha j} 2^{-3j\varepsilon /5 } \int_{t_1}^{t_0} 2^{-j\rho (Q)} \\
&\leq - c\, 2^{j(2\alpha - \rho(Q))} (t_0-t_1) (1-c\, 2^{-j\varepsilon /10 } ) ,
\end{split}
\]
where we used the lower bound $u_Q \geq 2^{-j\rho (Q)/2}/2$ (see \eqref{uQ_between_t1_t0}) in the last line. Therefore if $j_0$ is sufficiently large so that 
\[
1-c\, 2^{-j_1 \varepsilon /10 } >0
\]
(where $c$ is the last constant appearing in the calculation above; recall also that $j_1$ is given by \eqref{choice_of_j1}) we obtain
\[
1\leq 0,
\]
a contradiction.
\end{proof}
\begin{corollary}\label{cor_regularity_outside_E}
Given $x\not \in E$ and an interval of regularity $(a_i,b_i)$ there exists an open neighbourhood $U$ of $x$ such that
\[
\| u(t) \|_{L^{\infty } (U)} \text{ remains bounded for } t\in ((a_i+b_i)/2,b_i).
\] 
\end{corollary}
\begin{proof} We fix an interval of regularity. By Theorem \ref{thm_reg_outside_E} there exists $j_1$ and $r\in (0,2^{-10}) $ such that 
\[
u_Q (t) \leq 2^{-j\rho (Q)/2}
\]
for all $t\in [0,T)$ and all $j$-cubes $Q\subset rQ_{j_1} (x)$. Let $j_2\in \NN$ be the smallest number such that $\delta (Q) \geq 100/\varepsilon $ for every $j$-cube $Q\subset Q_{j_2}(x)$. (Note that the last condition implies also that $j\geq j_2$.) Then $\rho(Q) \geq 10 $ for any such $j$-cube $Q$ and so $u_Q \leq c\,2^{-5j}$. We let 
\[
U\coloneqq Q_{j_2+2}(x).
\]
In order to show that $\| u(t) \|_{L^\infty (U)}$ remains bounded, we note that the localised Bernstein inequality \eqref{localised_bernstein_ineq_single} gives
\[
\| \phi_Q P_j u \|_{\infty } \leq c\,2^{3j/2} u_Q +e(j)\leq c\, 2^{-7j/2} 
\]
for every $j$-cube $Q\in S_j (U)$ with $j\geq j_2+2$. Hence
\[
\| P_j u \|_{L^\infty (U) } \leq \sum_{Q\in S_j (U)}\| \phi_Q P_j u \|_{\infty } \leq c\, 2^{3(1-\varepsilon )(j-(j_2+2))}2^{-7j/2} =c_{j_2} 2^{-j/2}
\]
for such $j$ and so 
\[\begin{split}
\| u \|_{L^\infty (U) } &\leq \| P_{\leq j_2+1} u \|_{\infty  } +  \sum_{j\geq j_2+2} \| P_j u \|_{L^\infty (U) } \\
& \leq c\, 2^{3j_2/2} \| P_{\leq j_2+1} u \| +  c_{j_2} \sum_{j\geq j_2+2}  2^{-j/2} \\
& \leq c_{j_2} ,
\end{split}
\]
as required, where we used the Bernstein inequality \eqref{bernstein_ineq_leq} in the second inequality.
\end{proof}
\subsection{Regularity for $\alpha >5/4$}\label{sec_reg_alpha_large}
Here we briefly verify Corollary \ref{cor_alpha_big_and_higher_dim}. Letting $\varepsilon \in (0,4\alpha -5 )$ we see that any $j$-cube ($j\geq 0$) satisfies
\[
u_Q(t) \leq c \leq c \,2^{-j (5-4\alpha +\varepsilon ) }
\]
for all $t\geq 0$. Thus any closed and sufficiently small surface $\mathcal{S}\subset \RR^3$ can be used as a barrier, and Theorem \ref{thm_reg_outside_E} (with $\partial (r Q_{j_1}(x))$ replaced by $\mathcal{S}$) gives that $u_Q ( t) < 2^{-j \rho (Q)/2}$ for all $j$-cubes $Q$ located inside $\mathcal{S}$ and all $t\geq 0$ (provided $u_0$ is sufficiently smooth). Furthermore $j_2$ (from the proof of Corollary \ref{cor_regularity_outside_E}) can be chosen independently of $x$ (i.e. depending only on how small $\mathcal{S}$ is), and consequently Corollary \ref{cor_regularity_outside_E} gives boundedness of $\| u (t) \|_{\infty }$ in $t>0$.

\section{The box-counting dimension}\label{sec_db}

Here we prove Theorem~\ref{thm_main2}; namely that $d_B (S^{(k)}) \leq (-16\alpha^2+16\alpha +5 )/3$, where $S^{(k)} \coloneqq \bigcup_{i\leq k} S_i$ (recall \eqref{def_singset_k}). 

A bound on $d_B (S^{(k)})$ can in fact be obtained by examining the proof of Theorem~\ref{thm_reg_outside_E} above. Namely, observing that the only consequence of $x\not \in E$ that we used in its proof was that
\eqnb\label{used_for_naive_db}
x\not \in Q \text{ for any }Q \in \mathcal{B}_k, \,k\in [\theta^2 j_1 -10,j_1]
\eqne
where $j_1$ is taken sufficiently large. 
In fact, this allowed us to deduce that for a given $j$-cube $Q\subset rQ_{j_1} (x)$  the cube $Q_k=2^{(j-k)(1-\varepsilon )}Q$ is good for such $k$'s  (take $j_0\coloneqq \left\lfloor \theta^2 j_1-10 \right\rfloor$ and recall \eqref{barrier_conseq1}, \eqref{barrier_conseq2} and \eqref{q'_fact}). This, in the light of Theorem~\ref{thm_j_good_gives_critical regularity} gave us the ``slightly more than critical'' decay, which in turn enabled us to deduce better decay for cubes located further inside the barrier $rQ_{j_1}(x)$. Corollary~\ref{cor_regularity_outside_E} then deduced that $x\not \in S$.

Using \eqref{used_for_naive_db} we see that for sufficiently large $j$
\[
\bigcup_{k\in \{ \left\lfloor \theta^2 j-10 \right\rfloor , \ldots , j \} } \bigcup_{Q\in \mathcal{B}_k} Q
\]
contains the singular set in space at a given blow-up time. Thus, covering each of the covers $\mathcal{B}_k$ ($k\in \{ \left\lfloor \theta^2 j-10 \right\rfloor , \ldots , j \}$) by at most 
\[ c2^{3(j-k)(1-\varepsilon )} \# \mathcal{B}_k \leq  c2^{3(j-k)(1-\varepsilon )} 2^{k(5-4\alpha +\varepsilon )} = c2^{3j(1-\varepsilon )} 2^{k(2-4\alpha +2\varepsilon )}  \]
$j$-cubes we obtain a cover of the singular set by at most
\eqnb\label{atmost}\begin{split}
c \sum_{k=\left\lfloor \theta^2 j-10 \right\rfloor}^j 2^{3j(1-\varepsilon )} 2^{k(2-4\alpha +2\varepsilon )} &\leq c\,  2^{j(3-3\varepsilon +\theta^2 (2-4\alpha +2\varepsilon ))} \\
&=c\, 2^{j (-64 \alpha^3 + 96\alpha^2 (1+\varepsilon ) -48 \alpha   (1+\varepsilon )^2  + 35+8\varepsilon^3 + 8 \varepsilon^2 - 3\varepsilon )/9}
\end{split}
\eqne
$j$-cubes, where we substituted $\theta = 2(2\alpha -1 -\varepsilon ) /3$ (recall \eqref{def_of_theta}) in the last line. In other words $N(S^{(m)},r)$, the minimal number of $r$-balls required to cover $S^{(m)}$ (recall the definition \eqref{def_db} of the box-counting dimension), satisfies
\eqnb\label{est_on_N}
N(S^{(m)},r) \leq c\,r^{(-64 \alpha^3 + 96\alpha^2 (1+\varepsilon ) -48 \alpha   (1+\varepsilon )^2  + 35+8\varepsilon^3 + 8 \varepsilon^2 - 3\varepsilon )/9(1-\varepsilon )}
\eqne
for sufficiently small $r$.
This gives that
 \eqnb\label{dB_bound_not_sharp} d_B(S^{(m)})\leq (-64 \alpha^3 +96 \alpha^2 -48 \alpha +35 )/9 \eqne
for every $m\in \NN$. As noted in the introduction, we point out that the required smallness of $r$ for \eqref{est_on_N} to hold depends on the interval of regularity $(a_i,b_i)$. This is the reason why we only estimate $d_B(S^{(m)})$, rather than $d_B(S)$.\\

In what follows we present a sharper argument that allows one to get rid of one of $\theta$'s in the first line of \eqref{atmost} to yield the following.
\begin{proposition}\label{prop_db}
Given the interval of regularity $(a_i,b_i)$ the set 
\[
\bigcup_{k\in \{ \left\lfloor \theta j-10 \right\rfloor , \ldots , j \} } \bigcup_{Q\in \mathcal{B}_k} Q
\]
covers the singular set in space at time $b_i$ if $j$ is sufficiently large. 
\end{proposition}
Assuming this proposition and letting $\mathcal{C}_j$ be a $j$-cover of all elements of $\mathcal{B}_k$ for $k=\left\lfloor \theta j-10 \right\rfloor, \ldots , j$ we obtain a $j$-cover of the singular set with
\[\begin{split}
\# \mathcal{C}_j &\leq c \hspace{-0.2cm}\sum_{k= \left\lfloor \theta j-10 \right\rfloor}^j 2^{3(j-k)(1-\varepsilon )} \# \mathcal{B}_k \leq c\, 2^{j(3-3\varepsilon +\theta (2-4\alpha +2\varepsilon ))} = c\, 2^{j\frac{-16\alpha^2 +16\alpha (1+\varepsilon ) + 5 - 17\varepsilon - 4\varepsilon^2     }{3}},
\end{split}
\]
which shows that $d_B(S^{(m)}) \leq (-16\alpha^2 + 16\alpha +5)/3$ for all $m\in \NN$, by an analogous argument as above. This is sharper than \eqref{dB_bound_not_sharp}, and it proves Theorem \ref{thm_main2}. We note that if one was able to get rid of the other $\theta $ in \eqref{atmost}, then one would obtain $d_B(S)\leq 5-4\alpha $, i.e. the same bound as for $d_H(S)$.

Before proceeding to the proof of Proposition \ref{prop_db}, we comment on the main idea of Proposition~\ref{prop_db} in an informal way. 

Recall \eqref{temp_lk} that for each $k\in [\theta j , j-5]$ we needed $Q_{l-10}$ to be good for $l\in [\theta k, k]$, and deduced from the ``$\varepsilon$-better than critical'' decay for $u_{Q_k}$ (in Case~1 of Step~3 of the proof of Theorem~\ref{thm_reg_outside_E}, by using Theorem~\ref{thm_j_good_gives_critical regularity}), which we have then plugged into the sum of the low modes of the main estimate \eqref{basic_est_rewritten_for_full_reg} (in \eqref{where_to_plugin_low_modes} above). However, looking closely at this term of the main estimate,
\[
2^j \int_{t_1}^{t_0} u_Q u_{3Q/2, j\pm 2} \sum_{\theta j \leq k\leq j-5} 2^{3k/2} u_{Q_k}
\] 
we observe that it is of a similar structure as the definition of a good cube \eqref{j-good}. Indeed, ignoring $u_Q$ and $u_{3Q/2, j\pm 2}$ for a moment we see that we could use \eqref{j-good} to estimate it. If that was possible, we would only need to require that $Q_k$ (or rather $Q_{k-10}$) is good for $k\in [\theta j , j-5]$, and so we would end up with a saving of one $\theta$.
The only problem is that \eqref{j-good} is concerned with the time integral of a squared function, rather than the function itself, and so, applying the Cauchy-Schwarz inequality in the time integral we would obtain an additional factor of $(t_0-t_1)^{-1/2}$, see the last term in \eqref{db_new_plugging_in} below. It turns out that this additional factor can be taken care of, by absorbing a part of this term by the left-hand side (as in \eqref{db_absorbing} below).

\begin{proof}[Proof of Proposition~\ref{prop_db}.]
We will show that if $j_1$ is sufficiently large then every $x$ outside of $\mathcal{C}_{j_1}$ is a regular point in the given interval of regularity $(a,b)$. We set
\eqnb\label{db_def_of_j0}
j_0\coloneqq \left\lfloor \theta j_1 -10 \right\rfloor.
\eqne

As in Theorem~\ref{thm_reg_outside_E} we show that for sufficiently large $j_1=j_1(c_i)$ 
\eqnb\label{better_than_critical_db}
u_Q (t) < c_i 2^{-j\rho (Q)}
\eqne
for every $x\not \in \bigcup_{Q\in \mathcal{C}_{j_1}} Q$, where $c_i$ depends on the interval of regularity $(a_i,b_i)$. In fact, we can copy the entire proof of Theorem~\ref{thm_reg_outside_E}, except for Step~4, where we replace the estimate on the low modes (i.e. the first inequality in \eqref{est_terms_on_rhs}) by 
\eqnb\label{db_step4_replacement}
\sum_{k\in [\theta j , j-5]}  2^{3k/2} \int_{t_1}^{t_0} u_{Q_k} \leq c (t_0-t_1) 2^{-j(2-4\alpha +\varepsilon )/2} +  c (t_0-t_1)^{1/2}  2^{-j (2-2\alpha + \varepsilon )/2}  ,
\eqne
which we prove below. Given \eqref{db_step4_replacement}, we can plug it in, together with the remaining two inequalities in \eqref{est_terms_on_rhs}, into the main estimate \eqref{basic_est_rewritten_for_full_reg} (just as we did in Step~5 of the proof of Theorem~\ref{thm_reg_outside_E} above) to yield
\eqnb\label{db_new_plugging_in} \begin{split}
2^{-j\rho (Q) } &= c \left( u_Q(t_0 )^2 - u_Q(t_1)^2 \right) \\
&\leq  -c\,2^{2\alpha j }\int_{t_1}^{t_0}  u_{Q}^2+c \int_{t_1}^{t_0} u_{Q} \left( 2^j u_{3Q/2 ,j\pm 2} \sum_{\theta  j \leq k\leq j-5 } 2^{3k/2} u_{Q_k}   \right. \\
&\hspace{1.5cm}\left.+ 2^{5j/2}  u_{3Q/2 ,j\pm 4}^2+ 2^{3j/2} \sum_{k\geq j+1} 2^{k} u_{3Q/2 ,k}^2  \right) +2^{2\alpha j}2^{-j\varepsilon }\int_{t_1}^{t_0} u_{3Q/2,j\pm 2}^2 +e(j)\\
&\leq  -c\,2^{2\alpha j }(t_0-t_1) 2^{-j\rho (Q)} +c\,  2^{-j\rho (Q)} 2^{j(1+\varepsilon /5 )} \int_{t_1}^{t_0} \sum_{\theta  j \leq k\leq j-5 } 2^{3k/2} u_{Q_k}  \\
&\hspace{1.5cm}+ c (t_0-t_1) 2^{2\alpha j} 2^{-j\rho (Q)} \left( 2^{j\varepsilon /10} + 2^{j\varepsilon /10} + 2^{3j\varepsilon /5} \right) \\
& \leq 2^{2\alpha j }(t_0-t_1) 2^{-j\rho (Q)} \left( -c + c 2^{-j\varepsilon /10 } \right) + c(t_0-t_1 )^{1/2} 2^{-j\rho (Q) } 2^{\alpha j} 2^{-3j\varepsilon /10} 
\end{split}
\eqne
where, in the last step, we applied \eqref{db_step4_replacement} to estimate the low modes. At this point we obtain the same inequality as before (i.e. as in Step~5 of the proof of Theorem~\ref{thm_reg_outside_E}), except for the last term, which can be estimated using Young's inequality $ab\leq a^2/2+cb^2$ to give 
\eqnb\label{db_absorbing}
\frac12 2^{-j\rho(Q)} + c 2^{2\alpha j }(t_0-t_1) 2^{-j\rho (Q)} 2^{-3j\varepsilon /5}.
\eqne
Absorbing the first term above on the left-hand side we obtain
\[
1\leq 2^{2\alpha j }(t_0-t_1) \left( -c + c 2^{-j\varepsilon /10 } \right) ,
\]
which gives a contradiction for sufficiently large $j$.\\

It remains to verify \eqref{db_step4_replacement}. To this end, if $\delta (Q_k) \geq 11$ then, as before, we can use the fact that the claim \eqref{better_than_critical_db} remains valid until $t_0$ to obtain that
\[
\sum_{\substack{k\in [\theta j , j-5]\\ \delta (Q_k)\geq 11}} 2^{3k/2} \int_{t_1}^{t_0} u_{Q_k} \leq c(t_0-t_1)\sum_{\substack{k\in [\theta j , j-5]\\ \delta (Q_k)\geq 11}}  2^{3k/2} 2^{-k\rho (Q_k)/2}   \leq c (t_0-t_1) 2^{-j(2-4\alpha +\varepsilon )/2},
\]
where we used the fact that $\rho (Q_k)\geq 5-4\alpha +\varepsilon$ in the last inequality.

If $\delta (Q_k)\leq 10$ then $Q_{k-10}$ intersects the barrier $\p (rQ_{j_1} (x))$ and so it is good as $k-10 \geq \theta j -10 \geq j_0$ (recall \eqref{q'_fact} and \eqref{db_def_of_j0}). Thus since $\phi_{Q_k} \leq 1_{Q_{k-10}}$ (recall \eqref{what_is_phi_Q}) the definition \eqref{j-good} of a good cube gives
\[
\int_{t_1}^{t_0} u_{Q_k}^2 \leq \int_{t_1}^{t_0} \int_{Q_{k-10}} |P_k u |^2 \leq c\,2^{-k (5-2\alpha +\varepsilon )}.
\]
Hence
\[\begin{split}
\sum_{k\in [\theta j,  j-5], \delta (Q_k)\leq 10} 2^{3k/2} \int_{t_1}^{t_0}  u_{Q_k} &\leq (t_0-t_1)^{1/2} \sum_{\theta j, \ldots , j-5, \delta (Q_k)\leq 10 } 2^{3k/2} \left( \int_{t_1}^{t_0}   u_{Q_k}^2 \right)^{1/2} \\
&\leq c (t_0-t_1)^{1/2} \sum_{k\leq j-5} 2^{-k (2-2\alpha + \varepsilon )/2} \\
&= c (t_0-t_1)^{1/2}  2^{-j (2-2\alpha + \varepsilon )/2}  ,
\end{split}
\]
as required.
\end{proof}

\section*{Acknowledgements}
The majority of this work was conducted under postdoctoral funding from ERC 616797. The author has also been supported by the AMS Simons Travel Grant as well as funding from the Charles Simonyi Endowment at the Institute for Advanced Study. The author is grateful to Silja Haffter and Xiaoyutao Luo for their comments.
\bibliography{literature}{}

\end{document}